\documentclass[11pt]{amsart}

\issueinfo{00}{0}{Month}{Year}
\copyrightinfo{2008}{University of Calgary}
\pagespan{1}{2}
\PII{ISSN 1715-0868}
\include{cdmstdcmds}

\pdfoutput=1

\usepackage{amsfonts}
\usepackage{amssymb}
\usepackage{graphicx}
\usepackage{amsmath}
\usepackage{amscd}%
\usepackage{tikz}
\usepackage{tikz-cd}
\usetikzlibrary{decorations.markings}
\usetikzlibrary{arrows}
\usetikzlibrary{calc}
\providecommand{\U}[1]{\protect\rule{.1in}{.1in}}
\newtheorem{theorem}{Theorem}

\newtheorem{corollary}[theorem]{Corollary}

\newtheorem{definition}[theorem]{Definition}

\newtheorem{lemma}[theorem]{Lemma}

\newtheorem{proposition}[theorem]{Proposition}

\begin{document}

\title{Circuit partitions and signed interlacement in 4-regular graphs}
\author{Lorenzo Traldi}
\address{Lafayette College, Easton, Pennsylvania 18042}
\email{traldil@lafayette.edu}
\date{}

\begin{abstract}
Let $F$ be a 4-regular graph. Each circuit partition $P$ of $F$ has a
corresponding touch-graph $Tch(P)$; the circuits in $P$ correspond to vertices
of $Tch(P)$, and the vertices of $F$ correspond to edges of $Tch(P)$. We
discuss the connection between modified versions of the interlacement matrix
of an Euler system of $F$ and the cycle space of $Tch(P)$, over $GF(2)$ and
$\mathbb{R}$.
\end{abstract}
\maketitle
\section{Introduction}

This paper is concerned with the connection between two aspects of the structure of a 4-regular graph $F$: partitions of the edge set $E(F)$ into circuits, and interlacement of vertices with respect to Euler systems of $F$. We begin our discussion by reviewing some relevant background and terminology. 

The graphs we consider are unoriented multigraphs; loops and parallel edges are allowed. We think of every edge as consisting of two distinct half-edges, each half-edge incident on one vertex. The degree of a vertex is the number of incident half-edges, and a $d$-regular graph is one whose vertices all have degree $d$.

We use the term \emph{circuit} for an undirected closed trail. A circuit cannot traverse an edge more than once, but it may traverse a vertex more than once. An Euler circuit is a circuit that includes every edge of a graph; a familiar argument shows that a 4-regular graph $F$ has an Euler circuit if and only if $F$ is connected. Every 4-regular graph $F$ has an Euler system, i.e., a set that contains one Euler circuit for each connected component of $F$. 

We use the term \emph{circuit partition} for a partition of the edge set of a 4-regular graph into circuits (i.e., undirected closed trails). The idea of studying circuit partitions of 4-regular graphs was introduced by Kotzig \cite{K}, and developed further by Las Vergnas and Martin \cite{L2,L1, L, Ma}. Circuit partitions in 4-regular graphs have found applications and generalizations in Kauffman's bracket description of the Jones polynomial \cite{Kau}, and in the interlace polynomials of Arratia, Bollob\'{a}s and Sorkin \cite{A2, A}.

Some circuits in a 4-regular graph $F$ are illustrated in Figure \ref{circfig}. On the left we see an Euler circuit, $C$. To trace $C$, start at any vertex and follow edges around $F$, making sure to maintain the plain/dashed line status when passing through a vertex. (The plain/dashed status may change in the middle of an edge.) For instance, if we follow $C$ by starting at $a$ and walking to $b$ along the plain edge, we will encounter vertices in the order $abdabccda$. The same plain/dashed convention is used to indicate the circuits included in $P_1$ and $P_2$. The circuits in $P_1$ may be oriented to visit vertices in the orders $abdcba$, $ada$ and $cc$; the circuits in $P_2$ may be oriented to visit vertices in the orders $abda$ and $abccda$. N.b. Recall that our circuits have neither preferred starting points nor preferred directions; for instance, the longer circuit of $P_1$ might just as well be oriented to visit vertices in the order $cdbabc$.
\begin{figure} 
\centering
\begin{tikzpicture} [>=angle 90]
\draw [thick, fill=black] (-5,-1) circle (.13 cm);
\draw [thick, fill=black] (-5,1) circle (.13 cm);
\draw [thick, fill=black] (-3,-1) circle (.13 cm);
\draw [thick, fill=black] (-3,1) circle (.13 cm);
\draw [thick, dashed] (-3,1) to [out=-70,in=70] (-3,-1);
\draw [thick] (-3,1) to [out=-110,in=110] (-3,-1);
\draw [thick, dashed] (-3,1) to [out=160,in=20] (-5,1);
\draw [thick] (-3,1) to [out=200,in=-20] (-5,1);
\draw [thick, dashed] (-5,-1) -- (-3,-1);
\draw [thick, dashed, domain=90:270] plot ({-5+.3*cos(\x)}, {-1.3+.3*sin(\x)});
\draw [thick, domain=-90:90] plot ({-5+.3*cos(\x)}, {-1.3+.3*sin(\x)});
\draw [thick] (-5,-1) -- (-5,1);
\draw [thick, dashed] (-4,0) -- (-5,1);
\draw [thick] (-3,-1) -- (-4,0);
\node at (-2.5,.8) {$a$};
\node at (-2.5,-.8) {$b$};
\node at (-5.5,-.8) {$c$};
\node at (-5.5,.8) {$d$};
\draw [thick, fill=black] (-1,-1) circle (.13 cm);
\draw [thick, fill=black] (-1,1) circle (.13 cm);
\draw [thick, fill=black] (1,-1) circle (.13 cm);
\draw [thick, fill=black] (1,1) circle (.13 cm);
\draw [thick] (1,1) to [out=-70,in=90] (1.2,0);
\draw [thick, dashed] (1.2,0) to [out=-90,in=70] (1,-1);
\draw [thick] (1,1) to [out=-110,in=110] (1,-1);
\draw [thick, dashed] (1,1) to [out=160,in=20] (-1,1);
\draw [thick, dashed] (1,1) to [out=200,in=-20] (-1,1);
\draw [thick, dashed] (0,-1) -- (1,-1);
\draw [thick] (0,-1) -- (-1,-1);
\draw [thick, dashed, domain=0:360] plot ({-1+.3*cos(\x)}, {-1.3+.3*sin(\x)});
\draw [thick] (-1,-1) -- (-1,1);
\draw [thick] (1,-1) -- (-1,1);
\node at (1.5,.8) {$a$};
\node at (1.5,-.8) {$b$};
\node at (-1.5,-.8) {$c$};
\node at (-1.5,.8) {$d$};
\draw [thick, fill=black] (3,-1) circle (.13 cm);
\draw [thick, fill=black] (3,1) circle (.13 cm);
\draw [thick, fill=black] (5,-1) circle (.13 cm);
\draw [thick, fill=black] (5,1) circle (.13 cm);
\draw [thick, dashed] (5,1) to [out=-70,in=70] (5,-1);
\draw [thick] (5,1) to [out=-110,in=110] (5,-1);
\draw [thick, dashed] (5,1) to [out=160,in=20] (3,1);
\draw [thick] (5,1) to [out=200,in=-20] (3,1);
\draw [thick, dashed] (3,-1) -- (4,-1);
\draw [thick] (4,-1) -- (5,-1);
\draw [thick, dashed, domain=90:270] plot ({3+.3*cos(\x)}, {-1.3+.3*sin(\x)});
\draw [thick, domain=-90:90] plot ({3+.3*cos(\x)}, {-1.3+.3*sin(\x)});
\draw [thick] (3,-1) -- (3,1);
\draw [thick, dashed] (5,-1) -- (3,1);
\node at (5.5,.8) {$a$};
\node at (5.5,-.8) {$b$};
\node at (2.5,-.8) {$c$};
\node at (2.5,.8) {$d$};
\end{tikzpicture}
\caption{An Euler circuit $C$, a 3-element circuit partition $P_1$ and a 2-element circuit partition $P_2$. To follow a circuit, maintain the same plain/dashed line status when traversing a vertex.}
\label{circfig}
\end{figure}

Now, let $F$ be an arbitrary 4-regular graph. A \emph{transition} of $F$ at a vertex $v$ is a partition of the four half-edges incident at $v$ into two pairs; for instance each part of Figure \ref{circfig} indicates one transition at each vertex, with a pair of dashed half-edges and a pair of plain half-edges. $F$ has three different transitions at each vertex. An Euler system $C$ of $F$ may be used to label the transitions of $F$ in the following way. Temporarily choose an arbitrary orientation for each circuit included in $C$. Then for each vertex $v \in V(F)$, a person following the incident circuit of $C$ makes two ``entrances'' to $v$ and two ``exits'' from $v$; say entrance 1 is followed by exit 1, and entrance 2 is followed by exit 2. The ``entrances'' and ``exits'' are the four half-edges of $F$ incident at $v$. The transition that pairs entrance $i$ with exit $i$ for $i \in \{1,2\}$ is labeled $\phi_{C}(v)$; the transition that pairs entrance $i$ with exit $j$ for $i \neq j \in \{1,2\}$ is labeled $\chi_{C}(v)$; and the transition that pairs entrance 1 with entrance 2, and also pairs exit 1 with exit 2, is labeled $\psi_{C}(v)$. It is easy to see that each transition's label with respect to $C$ remains the same if the orientation of a circuit of $C$ is reversed. 

If $C$ and $C^{\prime}$ are different Euler systems of $F$ then some transitions will have different $\phi, \chi, \psi$ labels with respect to $C$ and $C^{\prime}$. For example, we leave it as an exercise for the reader to verify that in Figure \ref{circfig} there is an Euler circuit $C'$ of $F$ with  $\phi_{C'}(a)=\psi_{C}(a)$, $\phi_{C'}(b)=\chi_{C}(b)$, $\phi_{C'}(c)=\psi_{C}(c)$ and $\phi_{C'}(d)=\chi_{C}(d)$. Moreover, only two of the twelve transitions in $F$ have the same $\phi, \chi, \psi$ labels with respect to $C$ and $C'$.

It is easy to see that a circuit partition of $F$ is completely determined by choosing one transition at each vertex. For example, in Figure \ref{circfig} $P_1$ is determined by the transitions $\psi_{C}(a),\phi_{C}(b),\chi_{C}(c)$ and $\psi_{C}(d)$, while $P_2$ is determined by $\phi_{C}(a),\chi_{C}(b),\phi_{C}(c)$ and $\phi_{C}(d)$.

The notion of interlacement with respect to Euler systems in 4-regular graphs has been studied by many authors; see for instance \cite{Bold, CL, RR}.

\begin{definition}
\label{inter}If $C$ is an Euler system of $F$ then two vertices $v\neq w\in
V(F)$ are \emph{interlaced} with respect to $C$ if and only if there is a
circuit of $C$ on which $v$ and $w$ appear in the order $vwvw$ or $wvwv$. The
\emph{interlacement matrix} $\mathcal{I}(C)$ is the $V(F)\times V(F)$ matrix
with entries in the 2-element field $GF(2)$ given by: the $vw$ entry is $1$ if $v$ and $w$ are
interlaced, and $0$ otherwise.
\end{definition}

The fact that there is a connection between circuit partitions and interlacement has been discovered and rediscovered many times. Here is a statement that incorporates the versions of this connection that appear most often in the literature.

\begin{theorem}
\label{circuitnul}Suppose $C$ is an Euler system of a 4-regular graph $F$, and
$P$ is a circuit partition of $F$. Let $\mathcal{I}(C,P)$ be the
symmetric $GF(2)$-matrix obtained from $\mathcal{I}(C)$ by making these
two kinds of changes.
\begin{enumerate}
\item If $P$ involves the $\phi_{C}(v)$ transition, remove the row and column
corresponding to $v$.

\item If $P$ involves the $\psi_{C}(v)$ transition, change the $vv$ entry to
$1$.
\end{enumerate}

Then the $GF(2)$-nullity of $\mathcal{I}(C,P)$ is $\left\vert P\right\vert
-c(F)$, where $\left\vert P\right\vert $ is the number of circuits in $P$ and
$c(F)$ is the number of connected components in $F$.
\end{theorem}

We refer to the formula $\left\vert P\right\vert -c(F)=nullity(\mathcal{I}%
(C,P))$ as the \emph{circuit-nullity formula}. It seems that the earliest
discussion of some version of the formula appears in Brahana's 1921 study of
curves on surfaces \cite{Br}. However the formula was not widely known until
fifty years later, when a special case was discovered by Cohn and Lempel
\cite{CL}. Both of these references state versions of the circuit-nullity
formula which do not mention 4-regular graphs; Brahana refers to the
connectivity of a surface and Cohn and Lempel refer to the number of orbits in
a certain kind of permutation. Also, the version of Cohn and Lempel is
restricted to oriented Euler circuits and circuit partitions; the $\psi$
transitions are not relevant to the permutations they considered. Many other
authors have rediscovered, refined or restated the circuit-nullity formula in
various ways \cite{Be, BM, Bu, Br, CL, J1, Jo, KR, Lau, MP, Me, M, R, So, S,
Tbn, T5, Tnew, Z}.

We leave it as an exercise for the reader to confirm that the circuit-nullity formula holds in Figure \ref{circfig}, by calculating

\[
\mathcal{I}(C,P_1)=
\begin{pmatrix}
1 & 0 & 1 \\
0 & 0 & 0 \\
1 & 0 & 1
\end{pmatrix} \quad
\text{   and   } \quad \mathcal{I}(C,P_2)= (0).
\]

Another important part of the theory of circuit partitions is the notion of a touch-graph. This notion appeared implicitly in work of Jaeger \cite{J1},
and explicitly in Bouchet's work on isotropic systems
\cite{Bi1, Bi2}.

\begin{definition}
\label{touch}If $P$ is a circuit partition in a 4-regular graph $F$ then the
\emph{touch-graph} $Tch(P)$ has a vertex $v_{\gamma}$ for each circuit
$\gamma\in P$, and an edge $e_{v}$ for each vertex $v\in V(F)$; $e_{v}$ is
incident on $v_{\gamma}$ if and only if $\gamma$ passes through $v$.
\end{definition}

\begin{figure} 
\centering
\begin{tikzpicture} [>=angle 90]
\draw [thick, fill=black] (-4,0) circle (.13 cm);
\draw [thick] (-4,0) to [out=15,in=-45] (-3.2,0.8);
\draw [thick] (-4,0) to [out=60,in=135] (-3.2,0.8);
\draw [thick] (-4,0) to [out=165,in=-135] (-4.8,0.8);
\draw [thick] (-4,0) to [out=105,in=45] (-4.8,0.8);
\draw [thick] (-4,0) to [out=-15,in=45] (-3.2,-0.8);
\draw [thick] (-4,0) to [out=-75,in=-135] (-3.2,-0.8);
\draw [thick] (-4,0) to [out=195,in=135] (-4.8,-0.8);
\draw [thick] (-4,0) to [out=255,in=-45] (-4.8,-0.8);
\node at (-2.8,.8) {$e_a$};
\node at (-2.8,-.8) {$e_b$};
\node at (-5.25,-.8) {$e_c$};
\node at (-5.25,.8) {$e_d$};
\draw [thick, fill=black] (-1,-1) circle (.13 cm);
\draw [thick, fill=black] (0,0) circle (.13 cm);
\draw [thick, fill=black] (1,1) circle (.13 cm);
\draw [thick] (-1,-1) -- (0,0);
\draw [thick] (0,0) to [out=15,in=-105] (1,1);
\draw [thick] (0,0) to [out=75,in=-165] (1,1);
\draw [thick] (0,0) to [out=-15,in=45] (.7,-.7);
\draw [thick] (0,0) to [out=-75,in=-135] (.7,-.7);
\node at (1,.25) {$e_a$};
\node at (.3,-1) {$e_b$};
\node at (-.8,-.4) {$e_c$};
\node at (.2,.9) {$e_d$};
\draw [thick, fill=black] (3.5,-.5) circle (.13 cm);
\draw [thick, fill=black] (5,1) circle (.13 cm);
\draw [thick] (5,1) to [out=-90,in=0] (3.5,-.5);
\draw [thick] (5,1) to [out=-60,in=-30] (3.5,-.5);
\draw [thick] (5,1) to [out=-180,in=150] (3.5,-.5);
\draw [thick, domain=0:360] plot ({3.2+.4*cos(\x)}, {-.8+.4*sin(\x)});
\node at (5.1,-.2) {$e_a$};
\node at (4.2,0) {$e_b$};
\node at (2.5,-.8) {$e_c$};
\node at (3.7,.8) {$e_d$};
\end{tikzpicture}
\caption{Touch-graphs from Figure \ref{circfig}.}
\label{touchfig}
\end{figure}

The touch-graphs of the three circuit partitions of Figure \ref{circfig} are pictured in Figure \ref{touchfig}. 

\section{Statement of the main theorem}

Two questions about the circuit-nullity formula should come to mind.

Question 1. Is there a version of the circuit-nullity formula that involves
nullity over the reals instead of $GF(2)$?

Answer 1. Yes, but the real version that has appeared in the literature is of
limited generality. Brahana \cite{Br} discussed a skew-symmetric version of
his matrix for systems of curves drawn on two-sided surfaces, suggesting a
connection with topological orientability. Skew-symmetric versions of
$\mathcal{I}(C,P)$ have also been discussed by Bouchet \cite{Bu}, Jonsson
\cite{Jo}, Lauri \cite{Lau} and Macris and Pul\'{e} \cite{MP}. They all
require that $C$ and $P$ be orientation-consistent, i.e., $P$ cannot involve
any $\psi_{C}$ transition.

Question 2. Does the equality $nullity(\mathcal{I}(C,P))=\left\vert
P\right\vert -c(F)$ indicate a connection between $P$ and the null space of
$\mathcal{I}(C,P)$?

Answer 2. Yes, but for full generality the connection involves a non-symmetric
matrix in place of $\mathcal{I}(C,P)$. Building on earlier partial results
\cite{Bu, J1, T5}, we introduced a modified form of $\mathcal{I}(C,P)$ in
\cite{Tnew}, and showed that it is closely related to the touch-graph of $P$. This modified form of $\mathcal{I}(C,P)$ is defined as follows.

\begin{definition}
\label{modintmat}(\cite{Tnew}) Let $C$ be an Euler system of a 4-regular graph
$F$, and $P$ a circuit partition of $F$. Then the \emph{modified interlacement
matrix}$\ M(C,P)$ is the $V(F)\times V(F)$ matrix with entries in $GF(2)$
obtained from $\mathcal{I}(C)$ by making these two kinds of changes:

\begin{enumerate}
\item If $P$ involves the $\phi_{C}(v)$ transition, change the $vv$ entry to
$1$, and change every other entry of the $v$ column to $0$.

\item If $P$ involves the $\psi_{C}(v)$ transition, change the $vv$ entry to
$1$.
\end{enumerate}
\end{definition}

Observe that
\[
M(C,P)=%
\begin{pmatrix}
I & \ast\\
0 & \mathcal{I}(C,P)
\end{pmatrix}
\text{,}%
\]
where $I$ is an identity matrix whose rows and columns correspond to the
vertices of $F$ where $P$ involves the $\phi_{C}$ transition. It follows that
$M(C,P)$ has the same nullity as $\mathcal{I}(C,P)$. The main theorem of
\cite{Tnew} states that if we consider the rows of $M(C,P)$ as elements of the
vector space $GF(2)^{E(Tch(P))}$ instead of $GF(2)^{V(F)}$, then the
orthogonal complement of the row space of $M(C,P)$ is the subspace spanned by
the vertex cocycles of $Tch(P)$. (Recall that the cocycle of a vertex in a graph is the set of non-loop edges incident on that vertex.) To put it more simply: the row space of $M(C,P)$ is the cycle space of $Tch(P)$ over $GF(2)$.

As examples of this result from
\cite{Tnew}, consider that in Figure \ref{circfig} we have 
\[
M(C,P_1)=
\begin{pmatrix}
1 & 0 & 0 & 1 \\
1 & 1 & 0 & 1 \\
0 & 0 & 0 & 0 \\
1 & 0 & 0 & 1
\end{pmatrix}
\quad \text{and} \quad
M(C,P_2)=
\begin{pmatrix}
1 & 1 & 0 & 0 \\
0 & 0 & 0 & 0 \\
0 & 0 & 1 & 0 \\
0 & 1 & 0 & 1
\end{pmatrix}.
\]
The row space of $M(C,P_1)$ is generated by the first two rows, or equivalently, by $e_a+e_d$ (the first row) and $e_b$ (the difference between the first and second rows). The row space of $M(C,P_2)$ is generated by the three nonzero rows, or equivalently, by $e_a+e_b$, $e_c$ and $e_b+e_d$. Consulting Figure \ref{touchfig}, we see that these row spaces really do coincide with the cycle spaces of $Tch(P_1)$ and $Tch(P_2)$ over $GF(2)$.

Notice that the answers to Questions 1 and 2 are both of the form
\textquotedblleft Yes, but...\textquotedblright\ The second \textquotedblleft
but\textquotedblright\ is resolved over $GF(2)$ by using the nonsymmetric
matrix $M(C,P)$ in place of the traditional (skew-)symmetric $\mathcal{I}%
(C,P)$. The purpose of the present paper is to observe that the first
\textquotedblleft but\textquotedblright\ is also resolved by using
nonsymmetric matrices. In addition to determining the cycle space of $Tch(P)$ rather than only the size of $P$, our result is more general than
previously known versions of the circuit-nullity formula over $\mathbb{R}$; there is no orientability requirement.

\begin{theorem}
\label{main}Suppose $C$ is an Euler system of a 4-regular graph $F$, and $P$
is a circuit partition of $F$. Then there is a $V(F)\times V(F)$ matrix
$M_{\mathbb{R}}(C,P)$ with integer entries, with these two properties.

\begin{enumerate}
\item $M_{\mathbb{R}}(C,P)$ reduces to $M(C,P)$ (modulo $2$).

\item The row space of $M_{\mathbb{R}}(C,P)$ is the cycle space of $Tch(P)$
over $\mathbb{R}$.
\end{enumerate}
\end{theorem}

If $M_{\mathbb{R}}(C,P)$ satisfies Theorem \ref{main}, then $M_{\mathbb{R}%
}(C,P)$ also satisfies the circuit-nullity formula over $\mathbb{R}$; that is,
the $\mathbb{R}$-nullity of $M_{\mathbb{R}}(C,P)$ is $\left\vert P\right\vert
-c(F)$. The reason is simple:\ $M_{\mathbb{R}}(C,P)$ is a $V(F)\times V(F)$
matrix whose rank is the dimension of the cycle space of $Tch(P)$,%
\[
\left\vert E(Tch(P))\right\vert -\left\vert V(Tch(P))\right\vert
+c(Tch(P))=\left\vert V(F)\right\vert -\left\vert P\right\vert
+c(Tch(P))\text{.}%
\]
Consequently the $\mathbb{R}$-nullity of $M_{\mathbb{R}}(C,P)$ is $\left\vert
P\right\vert -c(Tch(P))$. It is easy to prove that $c(Tch(P))=c(F)$; see
Proposition \ref{comps} below.

Unless $Tch(P)$ is a forest, there are infinitely many different matrices
$M_{\mathbb{R}}(C,P)$ which satisfy\ Theorem \ref{main}. For if $M_{\mathbb{R}%
}(C,P)$ satisfies Theorem \ref{main} and $\rho$ is a nonzero row of
$M_{\mathbb{R}}(C,P)$, then we may add $\pm2\rho$ to any row of $M_{\mathbb{R}%
}(C,P)$ without disturbing either property specified in Theorem \ref{main}.
Because of this nonuniqueness we will often refer to \textquotedblleft an
$M_{\mathbb{R}}(C,P)$ matrix\textquotedblright\ rather than simply using the
notation $M_{\mathbb{R}}(C,P)$.

Theorem \ref{main} is proved in Section 3. In Section 4, we provide a standard
form for $M_{\mathbb{R}}(C,P)$, denoted $M_{\mathbb{R}}^{0}(C,P)$. The
standard form is defined using a signed version of $C$; that is, for each
$v\in V(F)$, one passage of a circuit of $C$ through $v$ is arbitrarily
designated $v^{+}$, and the other is $v^{-}$. When $C$ and $P$ respect the
same edge directions in $F$, $M_{\mathbb{R}}^{0}(C,P)$ is closely related to
the skew-symmetric matrices used by Bouchet \cite{Bu}, Jonsson \cite{Jo},
Lauri \cite{Lau} and Macris and Pul\'{e} \cite{MP}. Moreover, in this special
case $M_{\mathbb{R}}^{0}(C,P)$ has several attractive \textquotedblleft
naturality\textquotedblright\ properties; for instance if $C$ and $C^{\prime}$
are two Euler systems which respect the same edge directions then for each
signed version of $C$ there is a signed version of $C^{\prime}$ such that
$M_{\mathbb{R}}^{0}(C^{\prime},C)=M_{\mathbb{R}}^{0}(C,C^{\prime})^{-1}$. The
standard form does not have such nice properties in general. For instance, if
$C$ and $C^{\prime}$ are two Euler systems which do not respect the same edge
directions, then $M_{\mathbb{R}}^{0}(C,C^{\prime})^{-1}$ may have fractional
entries. An example of this type is presented in Section 5, along with a
couple of other examples; one of them shows that in general we cannot require
that $M_{\mathbb{R}}^{0}(C,P)$ be skew-symmetric. In Section 6 we discuss the
relationship between $M_{\mathbb{R}}(C,P)$ and $M_{\mathbb{R}}(C^{\prime},P)$
matrices, where $C$ and $C^{\prime}$ are Euler systems of $F$; we also
summarize the special features of the theory over $GF(2)$. In Sections 7 and
8 we discuss the special features of the orientation-consistent theory over $\mathbb{R}$, including the naturality properties mentioned earlier in this paragraph. The paper ends with a brief account of the important result of
Lauri \cite{Lau} and Macris and Pul\'{e} \cite{MP}, which gives a determinant
formula for the number of Euler systems of $F$ that respect the same edge
directions as $C$.

Before proceeding to give details, we should mention that the present paper provides the foundation for algebraic characterizations of circle graphs using multimatroid properties analogous to matroid regularity~\cite{BT4, BT}.

\section{Proof of the main theorem}

We begin with an elementary algebraic result. Let $f:\mathbb{Z}\rightarrow
GF(2)$ be the ring homomorphism with $f(1)=1$. If $G$ is a graph we obtain a
homomorphism $f:\mathbb{Z}^{E(G)}\rightarrow GF(2)^{E(G)}$ of abelian groups
by applying $f$ in each coordinate.

\begin{lemma}
\label{rank}If $S\subseteq\mathbb{Z}^{E(G)}$ then the rank of $S$ in
$\mathbb{R}^{E(G)}$ is not less than the rank of $f(S)$ in $GF(2)^{E(G)}$.
\end{lemma}

\begin{proof}
As the rank is the cardinality of a maximal linearly independent subset, it is
enough to show that if $T\subseteq S$ and $f(T)$ is linearly independent, then
$T$ is linearly independent too. Suppose instead that $T$ is linearly
dependent. Then there is a sum
\[
\sum_{t\in T}q_{t}t=0\text{,}%
\]
in which the coefficients $q_{t}$ are real numbers, not all of which are $0$.
Eliminating irrational factors, we may presume the $q_{t}$ are all rational;
then multiplying by their denominators and dividing by the greatest common
divisor, we may presume that the $q_{t}$ are integers whose g.c.d. is $1$. But
then
\[
\sum_{t\in T}f(q_{t})f(t)=0\text{,}%
\]
and the $f(q_{t})$ are not all $0$. This contradicts the independence of
$f(T)$.
\end{proof}

We take a moment to discuss our technical vocabulary. As mentioned in Section 1,  we think of an edge in a graph as consisting of two distinct half-edges, each half-edge incident on one vertex. When we want to direct an edge, we designate one of its half-edges as initial, and the other as terminal. Notice that this convention provides every edge with two distinct directions, even if the edge is a loop.

A directed walk in a graph is a sequence $W=v_{1}$, $h_{1}$, $h_{1}^{\prime}$, $v_{2}$,
..., $v_{k}$, $h_{k}$, $h_{k}^{\prime}$, $v_{k+1}$ such that for each $i$,
$h_{i+1}$ and $h_{i}^{\prime}$ are half-edges incident on $v_{i+1}$, and
$h_{i}$ and $h_{i}^{\prime}$ are the half-edges of an edge $e_{i}$. We consider the reversed sequence $W'=v_{k+1}$, $h_{k}^{\prime}$, $h_{k}$, $v_{k}$,
..., $v_{2}$, $h_{1}^{\prime}$, $h_{1}$, $v_{1}$ to define a different directed walk, even if $k=1$ and $e_1$ is a loop. However, $W$ and $W'$ define the same undirected walk. When we say ``$W$ is a walk'' without specifying that $W$ is directed, we usually mean that $W$ is undirected.

We take a moment to explain a special case. Suppose $W=v_{1}$, $h_{1}$, $h_{1}^{\prime}$, $v_{2}$,
..., $v_{k}$, $h_{k}$, $h_{k}^{\prime}$, $v_{k+1}$ is a directed walk with $k>1$, and there is an index $i$ such that $e_i$ is a loop. Then a new directed walk may be obtained from $W$ by interchanging $h_{i}$ and $h_{i}^{\prime}$. This new directed walk is distinct from $W$ because directed walks are sequences of half-edges. These two directed walks do not differ by simple reversal, so they define distinct undirected walks.

A trail is a walk without repeated edges, i.e., $e_{i}\neq e_{j}$ when $i\neq j\in\{1,...,k\}$. A path is a trail without repeated vertices except possibly at the beginning and end, i.e., $v_{i}\neq v_{j}$ when $i\neq j$ and
$\{i,j\}\neq\{1,k+1\}$.

A walk is closed
if $v_{1}=v_{k+1}$. We consider two closed directed walks to be the same if they
differ only by a cyclic permutation. That is, if $v_{1}=v_{k+1}$ then $v_{1}$,
$h_{1}$, $h_{1}^{\prime}$, $v_{2}$, ..., $v_{k}$, $h_{k}$, $h_{k}^{\prime}$,
$v_{k+1}$ and $v_{i}$, $h_{i}$, $h_{i}^{\prime}$, $v_{i+1}$, ...,
$h_{k}^{\prime}$, $v_{k+1}=v_{1}$, $h_{1}$, ..., $h_{i-1}^{\prime}$, $v_{i}$
determine the same closed directed walk. A closed trail is a circuit. (Some references agree with this usage, but others use \textquotedblleft
circuit\textquotedblright\ only for a closed path.) 

For notation and terminology regarding cycles and
cocycles in graphs, we follow Bollob\'{a}s \cite[Section II.\ 3]{B} for the
most part. We refer the reader there for proofs. Here is a summary.

Suppose $D$ is a directed version of a graph $G$ and $W$ is a directed walk in $G$. Let
$\mathbb{K}$ be a field, and $\mathbb{K}^{E(G)}$ the vector space over
$\mathbb{K}$ with basis $E(G)$. There is a vector $z_{D}(W)\in\mathbb{K}%
^{E(G)}$ determined by following $W$ from beginning to end, and for each
edge $e\in E(G)$, tallying $+1$ in the $e$ coordinate each time we pass
through $e$ in the $D$ direction, and $-1$ in the $e$ coordinate each time we
pass through $e$ in the opposite direction. The \emph{cycle space} $Z_{D}(G)$
over $\mathbb{K}$ is the subspace of $\mathbb{K}^{E(G)}$ spanned by
$\{z_{D}(W)\mid W$ is a closed directed walk in $G\}$. Also, if $X\subseteq V(G)$ then
there is an element $u_{D}(X)\in\mathbb{K}^{E(G)}$ whose $e$ coordinate, for
each $e\in E(G)$, is $+1$ if $e$ is directed in $D$ from a vertex in $X$ to a
vertex not in $X$, $-1$ if $e$ is directed in $D$ from a vertex not in $X$ to
a vertex in $X$, and $0$ otherwise. The subspace of $\mathbb{K}^{E(G)}$
spanned by $\{u_{D}(X)\mid X\subseteq V(G)\}$ is the \emph{cocycle space} of
$G$\ over $\mathbb{K}$, denoted $U_{D}(G)$.

We recall seven properties of these spaces. (i) No special property is
required of $\mathbb{K}$; any field will do. (However we are primarily
interested in $\mathbb{K}=GF(2)$ or $\mathbb{R}$.) (ii) No special property is
required of $D$; any directed version of $G$ yields spaces that correspond to
all closed walks and all cocycles. (iii) $Z_{D}(G)$ is spanned by the vectors
$z_{D}(W)$ such that $W$ is a minimal directed circuit. (iv) $U_{D}(G)$ is spanned by
the vectors $u_{D}(\{v\})$ such that $v\in V(G)$. (v) If $G$ has $c(G)$
connected components then the dimension of $U_{D}(G)$ is $\left\vert
V(G)\right\vert -c(G)$. (vi) $U_{D}(G)$ and $Z_{D}(G)$ are orthogonal
complements. (We refer to this property as \emph{cycle-cocycle duality}.)
(vii) The orthogonality between $U_{D}(G)$ and $Z_{D}(G)$ rests on the simple
observation that as we follow a closed directed walk, we must enter each subset
$X\subseteq V(G)$ the same number of times that we leave $X$. This simple
observation goes back to the very beginning of graph theory, in Euler's
discussion of the seven bridges of K\"{o}nigsberg.

The machinery of cycle-cocycle duality may be summarized in matrix form, like this:

\begin{theorem}
\label{cocycle}Given a spanning set $S$ for $Z_{D}(G)$, let $Z_{S}$ be the
$S\times E(G)$ matrix whose rows are the elements of $S$. Let $U_{V(G)}$ be
the $E(G)\times V(G)$ matrix whose columns are the vectors $u_{D}(\{v\})$,
$v\in V(G)$. Then the rank of $Z_{S}$ is $\left\vert
E(G)\right\vert -\left\vert V(G)\right\vert +c(G)$, the rank of $U_{V(G)}$ is
$\left\vert V(G)\right\vert -c(G)$, and $Z_{S}\cdot U_{V(G)}=0$.
\end{theorem}

We may now restate Theorem \ref{main} in the following equivalent form.

\begin{theorem}
\label{main2}Suppose $C$ is an Euler system of a 4-regular graph $F$, and $P$
is a circuit partition of $F$. Let $D$ be a directed version of $G=Tch(P)$.
Then there is a matrix $M_{\mathbb{R}}(C,P)$ of integers, which has the
following properties.

\begin{enumerate}
\item $M_{\mathbb{R}}(C,P)$ reduces (modulo $2$) to $M(C,P)$.

\item In the notation of Theorem \ref{cocycle} with $\mathbb{K}=\mathbb{R}$,
$Z_{D}(G)$ has a spanning set $S$ such that $Z_{S}=M_{\mathbb{R}}(C,P)$.
\end{enumerate}
\end{theorem}

Suppose now that $F$ is a 4-regular graph. As mentioned in Section 1, if $v \in V(F)$ then a transition at $v$ is a partition of the four half-edges of $F$ incident on $v$ into two pairs. Each of the pairs is called a \emph{single transition}. If $P$ is a circuit partition of a 4-regular graph $F$, then $P$ is determined by the choice of a transition
$P(v)$ at each vertex of $F$.

\begin{figure}[ptb]%
\centering
\includegraphics[
trim=2.943292in 8.158451in 2.676492in 1.739985in,
height=0.8579in,
width=2.1862in
]%
{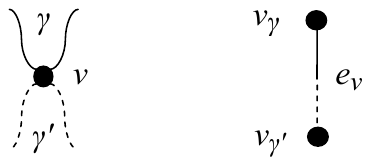}%
\caption{$F$ is on the left, and $Tch(P)$ is on the right.}%
\label{signinf1}%
\end{figure}

Recall that edges of $Tch(P)$ correspond to vertices of $F$ and vertices of
$Tch(P)$ correspond to circuits of $F$, as indicated in Figure \ref{signinf1}.
There is also a relationship between closed walks in $F$ and closed walks in
$Tch(P)$, which we proceed to describe.

As suggested in Figure \ref{signinf1}, there is a natural 2-to-1 surjection
\[
\pi_{P}:\{\text{half-edges of }F\}\rightarrow\{\text{half-edges of
}Tch(P)\}\text{,}%
\]
which we denote $\pi_{P}(h)=\overline{h}$. Suppose the four half-edges of $F$
incident on $v$ are $h_{v}^{1}$, $h_{v}^{2}$, $h_{v}^{3}$ and $h_{v}^{4}$, and
the two single transitions included in $P(v)$ are $\{h_{v}^{1},h_{v}^{2}\}$
and $\{h_{v}^{3},h_{v}^{4}\}$. Let $\gamma\in P$ be the circuit that includes
$h_{v}^{1}$ and $h_{v}^{2}$, and let $\gamma^{\prime}\in P$ be the circuit
that includes $h_{v}^{3}$ and $h_{v}^{4}$. Then the two half-edges of $e_{v}$
in $Tch(P)$ are $\overline{h_{v}^{1}}=\overline{h_{v}^{2}}$ and $\overline
{h_{v}^{3}}=\overline{h_{v}^{4}}$. The half-edge $\overline{h_{v}^{1}%
}=\overline{h_{v}^{2}}$ is incident on the vertex $v_{\gamma}\in V(Tch(P))$,
and the half-edge $\overline{h_{v}^{3}}=\overline{h_{v}^{4}}$ is incident on
the vertex $v_{\gamma^{\prime}}\in V(Tch(P))$.

This surjection $\pi_{P}$ on half-edges induces a related surjection,%
\[
\pi_{P}:\{\text{closed directed walks in }F\}\rightarrow\{\text{closed directed walks in
}Tch(P)\}\text{.}%
\]
Suppose $W$ is the closed directed walk $v_{1}$, $h_{1}$, $h_{1}^{\prime}$, $v_{2}$,
..., $v_{k}$, $h_{k}$, $h_{k}^{\prime}$, $v_{k+1}=v_{1}$ in $F$. Then there
are circuits $\gamma_{1},...,\gamma_{k}\in P$ such that $\gamma_{i}$ includes
the edge $e_{i}\in E(F)$ whose half-edges are $h_{i}$ and $h_{i}^{\prime}$.
Consider the list $v_{\gamma_{1}}$, $\overline{h_{1}^{\prime}}$,
$\overline{h_{2}}$, $v_{\gamma_{2}}$, ..., $v_{\gamma_{k}}$, $\overline
{h_{k}^{\prime}}$, $\overline{h_{1}}$, $v_{\gamma_{1}}$ of vertices and
half-edges in $Tch(P)$. Each index $i\in\{1,...,k\}$ is of one of the
following three types. A type (a) index has $\gamma_{i}\not =\gamma_{i+1}$. In
this case $\overline{h_{i}^{\prime}}\not =\overline{h_{i+1}}$ and $e_{v_{i+1}%
}=\{\overline{h_{i}^{\prime}},\overline{h_{i+1}}\}$ is a non-loop edge of
$Tch(P)$.\ A type (b) index has $\gamma_{i}=\gamma_{i+1}$, and the single
transition $\{h_{i}^{\prime},h_{i+1}\}$ is excluded from $P$. In this case
$\overline{h_{i}^{\prime}}\not =\overline{h_{i+1}}$ and $e_{v_{i+1}%
}=\{\overline{h_{i}^{\prime}},\overline{h_{i+1}}\}$ is a loop of $Tch(P)$. A
type (c) index has $\gamma_{i}=\gamma_{i+1}$, and the single transition
$\{h_{i}^{\prime},h_{i+1}\}$ is included in $P$. In this case $\overline
{h_{i}^{\prime}}=\overline{h_{i+1}}$ and the pair $\{\overline{h_{i}^{\prime}%
},\overline{h_{i+1}}\}$ is not an edge of $Tch(P)$. We define $\pi
_{P}(W)=\overline{W}$ to be the closed directed walk in $Tch(P)$ obtained from the list
$v_{\gamma_{1}}$, $\overline{h_{1}^{\prime}}$, $\overline{h_{2}}$,
$v_{\gamma_{2}}$, ..., $v_{\gamma_{k}}$, $\overline{h_{k}^{\prime}}$,
$\overline{h_{1}}$, $v_{\gamma_{1}}$ by removing every subsequence $\gamma
_{i}$, $\overline{h_{i}}$, $\overline{h_{i+1}^{\prime}}$ with $i$ of type (c).

\begin{proposition}
\label{comps}There is a one-to-one correspondence between connected components
of $F$ and $Tch(P)$: $\{v_{1},...,v_{k}\}$ is the vertex set of a connected
component of $F$ if and only if $\{e_{v_{1}},...,e_{v_{k}}\}$ is the edge set
of a connected component of $Tch(P)$.
\end{proposition}

\begin{proof}
As $F$ is 4-regular, every connected component of $F$ has an Euler circuit.
Two vertices of $F$ belong to the same connected component if and only if they
appear on the same one of these Euler circuits. The images of these Euler
circuits under $\pi_{P}$ are closed walks in $Tch(P)$, and two vertices of
$Tch(P)$ belong to the same connected component if and only if they appear on
the same one of these closed walks.
\end{proof}

\begin{definition}
Suppose $C$ is an Euler system for a 4-regular graph $F$, and $v\in V(F)$.
Then the \emph{induced circuits} of $C$ at $v$ are the two closed trails
obtained by following a circuit of $C$ from $v$ to $v$. We denote them
$C_{1}(C,v)$ and $C_{2}(C,v)$, with the indices arbitrary.
\end{definition}

That is, $\{C_{1}(C,v),C_{2}(c,v)\}$ is the circuit partition defined by
$\chi_{C}(v)$ and the transitions $\phi_{C}(w)$, $w\neq v$. The crucial
property of the induced circuits is this:

\begin{theorem}
\label{funcirc}Let $C$ be an Euler system for a 4-regular graph $F$, and let
$\Gamma$ be a set of induced circuits, which includes one of
$C_{1}(C,v),C_{2}(C,v)$ for each $v\in V(F)$. Choose either of the two directions for each $\gamma \in \Gamma$. Then for every circuit partition
$P$ of $F$ and every choice of a digraph $D$ on $Tch(P)$, the set
$S=\{z_{D}(\overline{\gamma})\mid\gamma\in\Gamma\}$ spans the subspace
$Z_{D}(Tch(P))$ of $\mathbb{R}^{E(Tch(P))}$.
\end{theorem}

\begin{proof}
Every $\gamma\in\Gamma$ is a directed closed walk in $F$, so $\overline{\gamma}$ is a directed
closed walk in $Tch(P)$. Consequently $S\subseteq Z_{D}(Tch(P))$. To prove
that $S$ spans $Z_{D}(Tch(P))$, it is enough to prove that the rank of $S$ is
at least
\[
\dim Z_{D}(Tch(P))=\left\vert E(Tch(P))\right\vert -\left\vert
V(Tch(P))\right\vert +c(Tch(P)).
\]

Let $f:\mathbb{Z}\rightarrow GF(2)$ be the map of Lemma \ref{rank}. Notice
that $M(C,P)$ is a $GF(2)$-matrix whose rows are the elements $f(s)$ with
$s\in S$, so the circuit-nullity formula over $GF(2)$ tells us that the
nullity of $f(S)$ is $\left\vert P\right\vert -c(F)=\left\vert
V(Tch(P))\right\vert -c(Tch(P))$. As $\left\vert S\right\vert =\left\vert
V(F)\right\vert =\left\vert E(Tch(P)\right\vert $, the rank of $f(S)$ is
$\left\vert f(S)\right\vert -nullity(f(S))=\left\vert E(Tch(P))\right\vert
-\left\vert V(Tch(P))\right\vert +c(Tch(P))$. The proof is completed by Lemma
\ref{rank}, which tells us that the rank of $S$ is not less than the rank of
$f(S)$.
\end{proof}

Theorem \ref{funcirc} tells us that if $\Gamma$ contains one directed induced
circuit for each vertex of $F$, then Theorems \ref{main} and \ref{main2} are
satisfied by the $V(F)\times V(F)$ matrix whose rows are the vectors
$z_{D}(\overline{\gamma})$, $\gamma\in\Gamma$.

\section{A standard form for $M_{\mathbb{R}}(C,P)$}

In this section we describe an $M_{\mathbb{R}}(C,P)$ matrix obtained by using
particular choices in the construction of Section 3. With these choices, all
entries of the matrix lie in $\{-1,0,1,2\}$. Moreover in the special case
involving orientation-consistent circuits, the matrix contains the
$\mathcal{I}(C,P)$ matrix used by Bouchet \cite{Bu}, Jonsson \cite{Jo}, Lauri
\cite{Lau} and Macris and Pul\'{e} \cite{MP}. More details about this special
case are given in Section 8.

Let $C$ be an Euler system of $F$. Arbitrarily choose preferred orientations for the circuits of $C$. For each $v\in V(F)$, let the half-edges of $F$
incident on $v$ be denoted $h_{v}^{1}$, $h_{v}^{2}$, $h_{v}^{3}$ and
$h_{v}^{4}$ in such a way that the circuit of $C$ incident on $v$ is
$...h_{v}^{1},v,h_{v}^{2},...,h_{v}^{3},v,h_{v}^{4},...$. As the incident circuit of $C$ does not have
a preferred starting point, the distinction between the two passages of $C$ through $v$ is arbitrary; we use $+$ and $-$ to distinguish them notationally:
one passage is $h_{v}^{1},v^{+},h_{v}^{2}$ and the other is $h_{v}^{3},v^{-},h_{v}^{4}$. Let $D$ be the directed version of $Tch(P)$ in which the
initial half-edge of the edge $e_{v}$ is $\overline{h_{v}^{1}}$. Index the
induced circuits $C_{1}(C,v),C_{2}(C,v)$ so that $C_{1}(C,v)$ includes $h_{1}^{v}$, and choose the preferred orientation of $C_{1}(C,v)$ consistent with the preferred orientation of the incident circuit of $C$. Let $M_{\mathbb{R}}^{0}(C,P)$ be the $V(F)\times V(F)$ matrix whose $v$
row is $z_{D}(\overline{C_{1}(C,v)})$, for each vertex $v$.

A compact way to encode this information is to write $C$ as a set of double occurrence words, one for each connected component of $F$, and for each vertex $v$, to designate which appearance is $v^{+}$ and which is $v^{-}$. Then for each $v\in V(F)$, the $v$ row of $M_{\mathbb{R}}^{0}(C,P)$ is obtained by tallying the contributions of passages through the vertices encountered as we follow the double occurrence word representing the incident circuit of $C$, from $v^{-}$ to $v^{+}$. We proceed to calculate the resulting entries $M_{\mathbb{R}}^{0}(C,P)_{vw}$.

Suppose $v\in V(F)$. The circuit $C_{1}(C,v)$ includes the passage $h_{v}^{1},v,h_{v}^{4}$ and no other passage through $v$. If $\phi _{C}(v)=P(v)$ then the initial half-edge of $e_{v}$ is $\overline{h_{v}^{1}}=\overline{h_{v}^{2}}$, and the terminal half-edge is $\overline{h_{v}^{3}}=\overline{h_{v}^{4}}$, so $\overline{C_{1}(C,v)}$ traverses $e_{v}$ in the
positive direction. If $\chi_{C}(v)=P(v)$ then the initial half-edge of $e_{v}$ is $\overline{h_{v}^{1}}=\overline{h_{v}^{4}}$, and the terminal
half-edge is $\overline{h_{v}^{2}}=\overline{h_{v}^{3}}$, so $\overline {C_{1}(C,v)}$ does not traverse $e_{v}$. If $\psi_{C}(v)=P(v)$ then the initial half-edge of $e_{v}$ is $\overline{h_{v}^{1}}=\overline{h_{v}^{3}}$,
and the terminal half-edge is $\overline{h_{v}^{2}}=\overline{h_{v}^{4}}$, so
$\overline{C_{1}(C,v)}$ traverses $e_{v}$ in the positive direction. We have the following.%

\[
M_{\mathbb{R}}^{0}(C,P)_{vv}=%
\begin{cases}
1\text{,} & \text{if }P(v)\in\{\phi_{C}(v),\psi_{C}(v)\}\\
0\text{,} & \text{if }P(v)=\chi_{C}(v)
\end{cases}
\]

Now, suppose $v\neq w\in V(F)$. If $\phi_{C}(w)=P(w)$ then any passage of
$C_{1}(C,v)$ through $w$ contributes $0$ to $M_{\mathbb{R}}^{0}(C,P)_{vw}$. If
$\chi_{C}(w)=P(w)$ then the initial half-edge of $e_{w}$ is $\overline
{h_{w}^{1}}=\overline{h_{w}^{4}}$, and the terminal half-edge is
$\overline{h_{w}^{2}}=\overline{h_{w}^{3}}$. Consequently if $\overline
{C_{1}(C,v)}$ includes the passage $\overline{h_{w}^{1}},\overline
{h_{w}^{2}}$ then this passage contributes $1$ to $M_{\mathbb{R}}%
^{0}(C,P)_{vw}$; and if $\overline{C_{1}(C,v)}$ includes the passage
$\overline{h_{w}^{3}},\overline{h_{w}^{4}}$ then this passage
contributes $-1$ to $M_{\mathbb{R}}^{0}(C,P)_{vw}$. If $\psi_{C}(w)=P(w)$ then
the initial half-edge of $e_{w}$ is $\overline{h_{w}^{1}}=\overline{h_{w}^{3}%
}$, and the terminal half-edge is $\overline{h_{w}^{2}}=\overline{h_{w}^{4}}$.
Consequently if $\overline{C_{1}(C,v)}$ includes the passage $\overline
{h_{w}^{1}},\overline{h_{w}^{2}}$ then this passage contributes $1$ to
$M_{\mathbb{R}}^{0}(C,P)_{vw}$; and if $\overline{C_{1}(C,v)}$ includes the
passage $\overline{h_{w}^{3}},\overline{h_{w}^{4}}$ then this passage
also contributes $1$ to $M_{\mathbb{R}}^{0}(C,P)_{vw}$. In sum, for $v \neq w \in V(F)$ we have the
following.
\begin{gather*}
M_{\mathbb{R}}^{0}(C,P)_{vw}=\\%
\begin{cases}
0\text{,} & \text{if }v\text{ and }w\text{ lie in different connected
components of }F\\
0\text{,} & \text{if }\phi_{C}(w)=P(w)\\
0\text{,} & \text{if }\chi_{C}(w)=P(w)\text{ and }v\text{ and }w\text{ are not
interlaced with respect to }C\\
1\text{,} & \text{if }\chi_{C}(w)=P(w)\text{ and a circuit of }C\text{ is
}v^{-}...w^{+}...v^{+}...w^{-}...\\
-1\text{,} & \text{if }\chi_{C}(w)=P(w)\text{ and a circuit of }C\text{ is
}v^{-}...w^{-}...v^{+}...w^{+}...\\
1\text{,} & \text{if }\psi_{C}(w)=P(w)\text{ and }v\text{ and }w\text{ are
interlaced with respect to }C\\
0\text{,} & \text{if }\psi_{C}(w)=P(w)\text{ and a circuit of }C\text{ is
}v^{-}...v^{+}...w...w...\\
2\text{,} & \text{if }\psi_{C}(w)=P(w)\text{ and a circuit of }C\text{ is
}v^{-}...w...w...v^{+}...
\end{cases}
\end{gather*}

The reader will have no trouble verifying the following properties of
$M_{\mathbb{R}}^{0}(C,P)$. Suppose we let $V(F)=V_{\phi}\cup V_{\chi}\cup
V_{\psi}$, in such a way that $v\in V_{\alpha}$ if and only if $\alpha
_{C}(v)=P(v)$. Then $M_{\mathbb{R}}^{0}(C,P)$ is%
\[
\bordermatrix{
& V_{\phi} & V_{\chi} & V_{\psi} \cr
V_{\phi} & I & M_1 & M_2 \cr
V_{\chi} & 0 & M_3 & M_4 \cr
V_{\psi} & 0 & M_5 & M_6 \cr
}\text{,}%
\]
where the indicated submatrices have the following properties. $I$ is an
identity matrix, the entries of $M_{1}$ all lie in $\{-1,0,1\}$, and the
entries of $M_{2}$ all lie in $\{0,1,2\}$. $M_{3}$ is a skew-symmetric matrix
with entries in $\{-1,0,1\}$. (In the special case $V_{\psi}=\allowbreak
\varnothing$, $M_{3}$ is the matrix $\mathcal{I}(C,P)$ used by Bouchet
\cite{Bu} (when $V_{\phi}$ is empty), Jonsson \cite{Jo}, Lauri \cite{Lau} and
Macris and Pul\'{e} \cite{MP}.) $M_{4}$ has entries from $\{0,1,2\}$ and
$M_{5}$ has entries from $\{-1,0,1\}$. There is a limited symmetry connecting
$M_{4}$ and $M_{5}$: if the $vw$ entry of $M_{4}$ is $0$ or $2$ then the $wv$
entry of $M_{5}$ is $0$; and if the $vw$ entry of $M_{4}$ is $1$ then the $wv$
entry of $M_{5}$ is $1$ or $-1$. $M_{6}$ has diagonal entries equal to $1$ and
all other entries from $\{0,1,2\}$; it reduces (mod $2$) to a symmetric
matrix. Interchanging the appearances of $v^{-}$ and $v^{+}$ on $C$ produces
three changes in $M_{\mathbb{R}}^{0}(C,P)$: if $P(v)=\chi_{C}(v)$ then the $v$
column of $M_{\mathbb{R}}^{0}(C,P)$ is multiplied by $-1$; if $P(w)=\chi
_{C}(w)$ then $M_{\mathbb{R}}^{0}(C,P)_{vw}$ is multiplied by $-1$; and if
$P(w)=\psi_{C}(w)$ then $M_{\mathbb{R}}^{0}(C,P)_{vw}$ is changed by the
replacement $0\leftrightarrow2$. Notice that all three changes have no effect
modulo $2$, reflecting the fact that $M(C,P)$ is a uniquely defined matrix
over $GF(2)$. Notice also that if $P$ does not involve any $\psi_{C}$
transition then the third kind of change does not occur, so the effect of
interchanging $v^{-}$ and $v^{+}$ on $C$ can be described using elementary row
and column operations; this special case is detailed in Section 8.

\section{Four examples}

Our first example illustrates the fact that if $C$ and $P$ do not respect the
same edge directions, it may be that there is no skew-symmetric matrix that
reduces to $\mathcal{I}(C,P)$ (mod $2$) and has nullity $\left\vert
P\right\vert -c(F)$.

Let $F$ be the 4-regular graph with $V(F)=\{a,b,c\}$ that is obtained from
$K_{3}$ by doubling edges. Then $F$ has an\ Euler circuit described by the
double occurrence word $abcabc$.\ We will use the standard form $M_{\mathbb{R}%
}^{0}(C,P)$ corresponding to $a^{+}b^{-}c^{+}a^{-}b^{+}c^{-}$, and the natural
notation for edges of $F$, e.g., the two edges connecting $a$ to $b$ are
$a^{+}b^{-}=b^{-}a^{+}$ and $a^{-}b^{+}=b^{+}a^{-}$. Let $P$ be the circuit
partition that includes $\gamma_{1}=\{a^{+}b^{-},a^{-}b^{+}\}$, $\gamma
_{2}=\{a^{+}c^{-},a^{-}c^{+}\}$ and $\gamma_{3}=\{b^{+}c^{-},b^{-}c^{+}\}$.
Then $Tch(P)\cong K_{3}$. Let $D$ be the oriented version of $Tch(P)$ used in
Section 4: $e_{a}$ is directed from $v_{\gamma_{2}}$ to $v_{\gamma_{1}}$,
$e_{b}$ is directed from $v_{\gamma_{1}}$ to $v_{\gamma_{3}}$ and $e_{c}$ is
directed from $v_{\gamma_{3}}$ to $v_{\gamma_{2}}$. Then $Z_{D}(Tch(P))$ is
spanned by the vector $(1,1,1)$.

$P$ involves the $\psi_{C}$ transition at every vertex, so
\[
\mathcal{I}(C,P)=M(C,P)=%
\begin{pmatrix}
1 & 1 & 1\\
1 & 1 & 1\\
1 & 1 & 1
\end{pmatrix}
.
\]
The $GF(2)$-nullity of $M(C,P)$ is 2, as predicted by the circuit-nullity
formula, and the rows of $M(C,P)$ span the cycle space $Z_{D}(Tch(P))$ over
$GF(2)$.

It is a simple matter to check that every skew-symmetric version of $M(C,P)$
is of nullity $0$ or $1$ over $\mathbb{R}$, so the circuit-nullity formula
over $\mathbb{R}$ is not satisfied by any skew-symmetric version of $M(C,P)$.
However the definition of Section 4 yields%

\[
M_{\mathbb{R}}^{0}(C,P)=%
\begin{pmatrix}
1 & 1 & 1\\
1 & 1 & 1\\
1 & 1 & 1
\end{pmatrix}
.
\]
The nullity of $M_{\mathbb{R}}^{0}(C,P)$ is $2$, and the row space of
$M_{\mathbb{R}}^{0}(C,P)$ is $Z_{D}(Tch(P))$.

Our second example illustrates Theorem \ref{main} for the standard form of
Section 4. Let $F$ be the simple 4-regular graph with
$V(F)=\{a,b,c,d,e,f,g,h\}$ and Euler circuit $C$ given by the signed double
occurrence word
\[
e^{-}a^{-}b^{-}f^{-}e^{+}h^{-}g^{-}f^{+}a^{+}d^{-}h^{+}c^{-}b^{+}g^{+}%
c^{+}d^{+}.
\]
Consider the circuit partition $P$ that involves the $\phi_{C}(a)$, $\chi
_{C}(e)$ and $\chi_{C}(g)$ transitions, along with the $\psi_{C}$ transition
at every other vertex. Then $P$ includes four circuits: $\gamma_{1}=\{ab$,
$bc$, $cd$, $da$, $af$, $fe$, $ea\}$, $\gamma_{2}=\{bf$, $fg$, $gb\}$,
$\gamma_{3}=\{ch$, $hg$, $gc\}$ and $\gamma_{4}=\{de$, $eh$, $hd\}$. The
construction of Section 4 yields the directed version of $Tch(P)$ illustrated
in Figure \ref{signinf3}, and the matrix
\[
M_{\mathbb{R}}^{0}(C,P)=%
\begin{pmatrix}
1 & 1 & 0 & 0 & 1 & 2 & -1 & 1\\
0 & 1 & 1 & 1 & 1 & 2 & -1 & 2\\
0 & 1 & 1 & 0 & 0 & 0 & 1 & 0\\
0 & 1 & 2 & 1 & 0 & 0 & 1 & 1\\
0 & 1 & 0 & 0 & 0 & 1 & 0 & 0\\
0 & 0 & 0 & 0 & 1 & 1 & -1 & 1\\
0 & 1 & 1 & 1 & 0 & 1 & 0 & 1\\
0 & 0 & 0 & 1 & 0 & 1 & -1 & 1
\end{pmatrix}
.
\]

\begin{figure}[ptb]%
\centering
\includegraphics[
trim=3.071593in 7.550942in 2.481066in 1.656343in,
height=1.5567in,
width=2.5287in
]%
{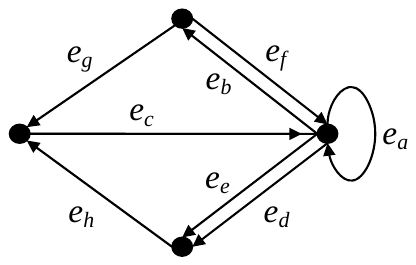}
\caption{The directed touch-graph from the second example.}%
\label{signinf3}%
\end{figure}

It is not hard to see that the rows of $M_{\mathbb{R}}^{0}(C,P)$ span the
cycle space of $Tch(P)$ over $\mathbb{R}$. Some rows represent individual
circuits, like $e_{b}+e_{f}$ (the fifth row) or $e_{e}+e_{f}-e_{g}+e_{h}$ (the
sixth row); other rows represent combinations of circuits, like $e_{a}%
+(e_{b}+e_{f})+\ (e_{e}+e_{f}-e_{g}+e_{h})$ (the first row). Also,
$M_{\mathbb{R}}^{0}(C,P)$ reduces to $M(C,P)$ (mod 2), and the product
$M_{\mathbb{R}}^{0}(C,P)\cdot U_{V(Tch(P))}$ is
\[%
\begin{pmatrix}
1 & 1 & 0 & 0 & 1 & 2 & -1 & 1\\
0 & 1 & 1 & 1 & 1 & 2 & -1 & 2\\
0 & 1 & 1 & 0 & 0 & 0 & 1 & 0\\
0 & 1 & 2 & 1 & 0 & 0 & 1 & 1\\
0 & 1 & 0 & 0 & 0 & 1 & 0 & 0\\
0 & 0 & 0 & 0 & 1 & 1 & -1 & 1\\
0 & 1 & 1 & 1 & 0 & 1 & 0 & 1\\
0 & 0 & 0 & 1 & 0 & 1 & -1 & 1
\end{pmatrix}
\cdot%
\begin{pmatrix}
0 & 0 & 0 & 0\\
1 & -1 & 0 & 0\\
-1 & 0 & 1 & 0\\
1 & 0 & 0 & -1\\
1 & 0 & 0 & -1\\
-1 & 1 & 0 & 0\\
0 & 1 & -1 & 0\\
0 & 0 & -1 & 1
\end{pmatrix}
=0.
\]
Notice that if we add $-2$ times the third row of $M_{\mathbb{R}}^{0}(C,P)$ to
the fourth row, and add $-2$ times the sixth row to each of the first two
rows, then the resulting matrix has the same reduction (mod $2$) and the same
row space as $M_{\mathbb{R}}^{0}(C,P)$, and its entries are all in
$\{-1,0,1\}$. We do not know whether it is always possible to eliminate
entries outside $\{-1,0,1\}$ in this way.

Our third example involves two Euler circuits of $K_{5}$: $C$ is given by the
double occurrence word $abdcaecbed$ and $C^{\prime}$ is given by the double
occurrence word $abcdecadbe$.\ On the left below is the $M_{\mathbb{R}}%
^{0}(C,C^{\prime})$ matrix for the signed version $a^{-}b^{-}d^{-}c^{-}%
a^{+}e^{-}c^{+}b^{+}e^{+}d^{+}$ of $C$; its inverse appears on the right. (We
abuse notation slightly by writing $M_{\mathbb{R}}^{0}(C,C^{\prime})$ rather
than $M_{\mathbb{R}}^{0}(C,\{C^{\prime}\})$.)
\[%
\begin{pmatrix}
{\ 1} & {1} & {-1} & {-1} & {0}\\
{\ 1} & {1} & {\ 0} & {-1} & {1}\\
{\ 1} & {0} & {\ 0} & {\ 0} & {1}\\
{\ 1} & {1} & {\ 0} & {\ 0} & {2}\\
{\ 0} & {1} & {\ 1} & {\ 0} & {1}%
\end{pmatrix}
^{-1}=%
\begin{pmatrix}
1 & -1 & 2 & -1 & 1\\
1 & -1 & 0 & 0 & 1\\
0 & 0 & 1 & -1 & 1\\
1 & -2 & 1 & 0 & 1\\
-1 & 1 & -1 & 1 & -1
\end{pmatrix}
\]
The inverse matrix is an $M_{\mathbb{R}}(C^{\prime},C)$ matrix, though it is
not in standard form. The $M_{\mathbb{R}}^{0}(C,C^{\prime})$ matrix for the
signed version $a^{-}b^{+}d^{+}c^{-}a^{+}e^{-}c^{+}b^{-}e^{+}d^{-}$ of $C$ is
on the left below; its inverse appears on the right.%

\[%
\begin{pmatrix}
{\ 1} & {1} & {-1} & {\ 1} & {0}\\
{\ 1} & {1} & {\ 0} & {-1} & {\ 1}\\
{\ 1} & {0} & {\ 0} & {\ 0} & {\ 1}\\
{\ 1} & {1} & {\ 0} & {\ 0} & {\ 0}\\
{\ 0} & {1} & {\ 1} & {\ 0} & {\ 1}%
\end{pmatrix}
^{-1}=\left(  \frac{1}{3}\right)
\begin{pmatrix}
-1 & -1 & 2 & 3 & -1\\
1 & 1 & -2 & 0 & 1\\
-2 & -2 & 1 & 3 & 1\\
1 & -2 & 1 & 0 & 1\\
1 & 1 & 1 & -3 & 1
\end{pmatrix}
\allowbreak\allowbreak
\]
$\allowbreak$In this case $M_{\mathbb{R}}^{0}(C,C^{\prime})^{-1}$ is not a
matrix of integers, so it is certainly not an $M_{\mathbb{R}}(C^{\prime},C)$
matrix; but $3\cdot M_{\mathbb{R}}^{0}(C,C^{\prime})^{-1}$ is an
$M_{\mathbb{R}}(C^{\prime},C)$ matrix. Also, $\det M_{\mathbb{R}}%
^{0}(C,C^{\prime})=3$ tells us that the rows of $M_{\mathbb{R}}^{0}%
(C,C^{\prime})$ generate a proper subgroup of $\mathbb{Z}^{E(Tch(C^{\prime}%
))}$. Every edge of $Tch(C^{\prime})$ is a loop, though, so the cycle space of
$Tch(C^{\prime})$ includes all of $\mathbb{Z}^{E(Tch(C^{\prime}))}$.

Our fourth example includes $C^{\prime}$ and another Euler circuit
$C^{\prime\prime}$ of $K_{5}$, given by the double occurrence word
$abecdbcade$. Using the signed form $a^{-}b^{+}c^{-}d^{+}e^{-}c^{+}a^{+}%
d^{-}b^{-}e^{+}$ of $C^{\prime}$, we obtain%

\[
M_{\mathbb{R}}^{0}(C^{\prime},C^{\prime\prime})=%
\begin{pmatrix}
{\ 1} & {1} & {\ 0} & {1} & {-1}\\
{\ 0} & {0} & {\ 0} & {0} & {\ 1}\\
{\ 0} & {0} & {\ 0} & {1} & {-1}\\
{\ 0} & {0} & {-1} & {0} & {\ 1}\\
{\ 0} & {-1} & {\ 1} & {-1} & {\ 0}%
\end{pmatrix}
\text{.}%
\]
Using the signed form $a^{+}b^{-}e^{-}c^{-}d^{-}b^{+}c^{+}a^{-}d^{+}e^{+}$ of
$C^{\prime\prime}$, we obtain%
\[
M_{\mathbb{R}}^{0}(C^{\prime\prime},C^{\prime})=%
\begin{pmatrix}
1 & 0 & 0 & 1 & 1\\
0 & 0 & -1 & -1 & -1\\
0 & 1 & 0 & -1 & 0\\
0 & 1 & 1 & 0 & 0\\
0 & 1 & 0 & 0 & 0
\end{pmatrix}
=M_{\mathbb{R}}^{0}(C^{\prime},C^{\prime\prime})^{-1}\text{.}%
\]

\section{The effect of a $\kappa$-transformation}

The fundamental operation of the theory of Euler systems of 4-regular graphs
was introduced by Kotzig \cite{K}.

\begin{definition}
\label{kappat}If $C$ is an\ Euler system of a 4-regular graph $F$ and $v\in
V(F)$ then a $\kappa$\emph{-transform} $C\ast v$ is an Euler system obtained
from $C$ by reversing one of the induced circuits $C_{i}(C,v)$ within a
circuit of $C$.
\end{definition}

If $C$ is given with preferred orientations for its circuits, then Definition \ref{kappat} provides two choices for the preferred orientation of the circuit of $C\ast v$ incident at $v$. For instance, if $C$ is the directed Euler circuit of $K_{5}$ given by the double occurrence word $abdcaecbed$ then $C\ast a$ is given by $acdbaecbed$ or $abdcadebce$. 

Kotzig \cite{K} proved that if $C$ and $C^{\prime}$ are any two\ Euler systems of $F$ then there is a
sequence $v_{1},...,v_{k}$ of vertices of $F$ such that $C^{\prime}=C\ast
v_{1}\ast\cdots\ast v_{k}$. We refer to this fundamental result as
\emph{Kotzig's theorem}.

It is not hard to see that the effect of a $\kappa$-transformation on
transition labels is given by the following.

\begin{proposition}
\label{kappalabel}Transition labels with respect to $C$ and $C\ast v$ differ
only in these two ways.

\begin{itemize}
\item $\phi_{C}(v)=\psi_{C\ast v}(v)$ and $\psi_{C}(v)=\phi_{C\ast v}(v)$.

\item If $w$ is interlaced with $v$ then $\chi_{C}(w)=\psi_{C\ast v}(w)$ and
$\psi_{C}(w)=\chi_{C\ast v}(w)$.
\end{itemize}
\end{proposition}

Recall that if we are given $C$ and $P$, $M(C,P)$ is the matrix over $GF(2)$
specified in Definition \ref{modintmat}. Proposition \ref{kappalabel} implies
the following three properties, which we describe collectively as
\textquotedblleft naturality\textquotedblright\ of $M(C,P)$\ with respect to
$\kappa$-transformations. See \cite{Tnew} for a detailed discussion. (Special
cases of the third property appear also in earlier work of Bouchet \cite{Bu}
and Jaeger \cite{J1}.)

\begin{corollary}
\label{kappachange}(\cite{Tnew}) If $P$ is a circuit partition of $F$ and
$C,C^{\prime}$ are Euler systems of $F$ then the following properties hold.

\begin{enumerate}
\item If $v\in V(F)$ then $M(C\ast v,P)$ is obtained from $M(C,P)$ by adding
the $v$ row to the $w$ row whenever $w\neq v$ and $w$\ is interlaced with $v$
on $C$.

\item $M(C^{\prime},P)=M(C^{\prime},C)\cdot M(C,P)$.

\item $M(C,C^{\prime})=M(C^{\prime},C)^{-1}$.
\end{enumerate}
\end{corollary}

\begin{proof}
The first assertion follows from Proposition \ref{kappalabel}.

For the second property recall that by Kotzig's theorem, there is a sequence
$v_{1},...,v_{k}$ of vertices of $F$ such that $C^{\prime}=C\ast v_{1}%
\ast\cdots\ast v_{k}$. The first property tells us that this sequence of
$\kappa$-transformations induces a corresponding sequence of elementary row
operations, which transforms the double matrix%
\[%
\begin{pmatrix}
I=M(C,C) & M(C,P)
\end{pmatrix}
\]
into the double matrix%
\[%
\begin{pmatrix}
M(C^{\prime},C) & M(C^{\prime},P)
\end{pmatrix}
\text{.}%
\]
It follows that if $E$ is the product of elementary matrices corresponding to
the induced elementary row operations, then $E\cdot I=M(C^{\prime},C)$ and
$E\cdot M(C,P)=M(C^{\prime},P)$.

For the third property, notice that the second property tells us that
$I=M(C^{\prime},C^{\prime})=M(C^{\prime},C)\cdot M(C,C^{\prime})$.
\end{proof}

Over $\mathbb{R}$, in contrast, we do not have a uniquely defined
$M_{\mathbb{R}}(C,P)$ matrix. Consequently the naturality\ properties of
$M_{\mathbb{R}}(C,P)$ over $\mathbb{R}$ are less precise than the properties
of Corollary \ref{kappachange}.

\begin{corollary}
\label{realchange} If $C$ and $C^{\prime}$ are Euler systems of $F$ then the
following properties hold.

\begin{enumerate}
\item Every $M_{\mathbb{R}}(C,C^{\prime})$ matrix is nonsingular, and has the
property that%
\[
\left(  \det M_{\mathbb{R}}(C,C^{\prime})\right)  \cdot M_{\mathbb{R}%
}(C,C^{\prime})^{-1}%
\]
is an $M_{\mathbb{R}}(C^{\prime},C)$ matrix.

\item Let $P$ be a circuit partition of $F$. Given an $M_{\mathbb{R}%
}(C^{\prime},C)$ matrix and an $M_{\mathbb{R}}(C,P)$ matrix, the product%
\[
M_{\mathbb{R}}(C^{\prime},C)\cdot M_{\mathbb{R}}(C,P)
\]
is an $M_{\mathbb{R}}(C^{\prime},P)$ matrix.
\end{enumerate}
\end{corollary}

\begin{proof}
As $M_{\mathbb{R}}(C,C^{\prime})$ satisfies Theorem \ref{main}, it is a
nonsingular matrix of integers that reduces to $M(C,C^{\prime})$ (mod $2$); it
follows that $\det M_{\mathbb{R}}(C,C^{\prime})$ reduces to $\det
M(C,C^{\prime})$ (mod $2$). The circuit-nullity formula tells us that
$M(C,C^{\prime})$ is a nonsingular $GF(2)$-matrix, so $\det M_{\mathbb{R}%
}(C,C^{\prime})$ is an odd integer. It follows that $\left(  \det
M_{\mathbb{R}}(C,C^{\prime})\right)  \cdot M_{\mathbb{R}}(C,C^{\prime})^{-1}$
is a nonsingular matrix of integers that reduces (mod $2$) to $M(C,C^{\prime
})^{-1}$. Corollary \ref{kappachange} tells us that $M(C,C^{\prime}%
)^{-1}=M(C^{\prime},C)$, so $\left(  \det M_{\mathbb{R}}(C,C^{\prime})\right)
\cdot M_{\mathbb{R}}(C,C^{\prime})^{-1}$ is an $M_{\mathbb{R}}(C^{\prime},C)$ matrix.

For the second property, notice that the nonsingularity of $M_{\mathbb{R}%
}(C^{\prime},C)$ implies that the row space of $M_{\mathbb{R}}(C^{\prime
},C)\cdot M_{\mathbb{R}}(C,P)$ is the same as the row space of $M_{\mathbb{R}%
}(C,P)$. Corollary \ref{kappachange} tells us that $M_{\mathbb{R}}(C^{\prime
},C)\cdot M_{\mathbb{R}}(C,P)$ reduces to $M(C^{\prime},P)$ (mod $2$), so
$M_{\mathbb{R}}(C^{\prime},C)\cdot M_{\mathbb{R}}(C,P)$ is an $M_{\mathbb{R}%
}(C^{\prime},P)$ matrix.
\end{proof}

Multiplying by $\det M_{\mathbb{R}}(C,C^{\prime})$ is necessary in part 1
because as we saw in Section 5, if $\left \vert \det M_{\mathbb{R}}(C,C^{\prime}) \right \vert>1$ then $M_{\mathbb{R}}(C,C^{\prime})^{-1}$ may have entries that are not integers. 

\section{The effect of a transposition}

In addition to $\kappa$-transformations, Kotzig \cite{K} also defined
\textquotedblleft$\varrho$-transforma-tions\textquotedblright\ on Euler
systems. We follow Arratia, Bollob\'{a}s and Sorkin \cite{A2, A} and use a
different name for this operation.

\begin{definition}
\label{pivot}If $C$ is an\ Euler system of a 4-regular graph $F$ and $v,w\in
V(F)$ are interlaced with respect to $C$, then the \emph{transposition}
$C\ast(vw)$ is an Euler system obtained from $C$ by interchanging the
$v$-to-$w$ trails within a circuit of $C$.
\end{definition}

Several properties of transpositions are readily apparent. One property is that if the circuits of $C$ are given with preferred orientations, then the circuits of $C\ast(vw)$ inherit preferred orientations in a natural way. Also, the transpositions $C\ast(vw)$ and $C\ast(wv)$ are the same. Moreover, a transposition can be effected by performing three $\kappa$-transformations: if
$C=vT_{1}wT_{2}vT_{3}wT_{4}$ then
\begin{align*}
((C\ast v)\ast w)\ast v  &  =((v\overleftarrow{T_{2}}w\overleftarrow{T_{1}%
}vT_{3}wT_{4})\ast w)\ast v\\
&  =(v\overleftarrow{T_{2}}w\overleftarrow{T_{3}}vT_{1}wT_{4})\ast
v=vT_{3}wT_{2}vT_{1}wT_{4}=C\ast(vw)\text{,}%
\end{align*}
where $\overleftarrow{T_{i}}$ indicates reversal of the trail $T_{i}$. Another
property is that $C$ and $C\ast(vw)$ respect the same edge directions. In
fact, Kotzig \cite{K}, Pevzner \cite{P} and Ukkonen \cite{U} proved that if
$C$ and $C^{\prime}$ are two Euler systems of $F$, then $C$ and $C^{\prime}$
respect the same edge directions if and only if it is possible to obtain
$C^{\prime}$ from $C$ using transpositions.

It is not hard to see that the effect of a transposition on transition labels
is given by the following.

\begin{proposition}
\label{pivotlabel}If $v$ and $w$ are interlaced with respect to $C$ then
transition labels with respect to $C$ and $C\ast(vw)$ differ only in these
ways: $\chi_{C\ast(vw)}(v)=\phi_{C}(v)$, $\phi_{C\ast(vw)}(v)=\chi_{C}(v),$
$\chi_{C\ast(vw)}(w)=\phi_{C}(w)$ and $\phi_{C\ast(vw)}(w)=\chi_{C}(w)$.
\end{proposition}

Despite the fact that a transposition's effect on transition labels is less
complicated than the effect of a $\kappa$-transformation, Euler systems
related through transpositions may give rise to $M_{\mathbb{R}}^{0}$ matrices
that are related in complicated ways. For example, the following Euler
circuits of $K_{5}$ yield the matrices below.
\begin{align*}
C  &  \text{:}\text{ }a^{-}e^{-}c^{+}b^{+}d^{+}c^{-}a^{+}b^{-}e^{+}d^{-}\\
C\ast(cd)\text{ }\text{: }  &  a^{-}e^{-}c^{+}a^{+}b^{-}e^{+}d^{-}c^{-}%
b^{+}d^{+}\\
C^{\prime}\text{ }\text{: }  &  abcdecadbe
\end{align*}%
\[
M_{\mathbb{R}}^{0}(C,C^{\prime})=%
\begin{pmatrix}
{1} & {1} & {0} & {1} & {1}\\
{1} & {1} & {1} & {-1} & {2}\\
{2} & {1} & {0} & -{1} & {2}\\
{1} & {1} & {1} & {0} & {1}\\
{1} & {2} & {0} & {1} & {1}%
\end{pmatrix}
\]

\[
M_{\mathbb{R}}^{0}(C\ast(cd),C^{\prime})=%
\begin{pmatrix}
{1} & {0} & {0} & {0} & {1}\\
{0} & {1} & {0} & {0} & {1}\\
{1} & {1} & {1} & 0 & {1}\\
{0} & {1} & {0} & {1} & {0}\\
{1} & {1} & {0} & {0} & {1}%
\end{pmatrix}
\]
Notice that $\det M_{\mathbb{R}}^{0}(C,C^{\prime})=-3$ and $\det
M_{\mathbb{R}}^{0}(C\ast(cd),C^{\prime})=-1$, so although $M_{\mathbb{R}}%
^{0}(C,C^{\prime})$ and $M_{\mathbb{R}}^{0}(C\ast(cd),C^{\prime})$ are row
equivalent over $\mathbb{R}$, they are not row equivalent over $\mathbb{Z}$.

\section{The oriented case}

In this section we show that in case $C$ and $P$ respect the same edge
directions, the standard form $M_{\mathbb{R}}^{0}(C,P)$ described in Section
4 has naturality properties over $\mathbb{Z}$ that are very similar to the
naturality properties of $M(C,P)$ over $GF(2)$, stated in Corollary
\ref{kappachange}. Moreover, $M_{\mathbb{R}}^{0}(C,P)$ includes the
skew-symmetric signed interlacement matrices of Brahana \cite{Br}, Bouchet
\cite{Bu}, Jonsson \cite{Jo}, Lauri \cite{Lau} and Macris and Pul\'{e} \cite{MP}.

Suppose $C$ is a directed Euler system of $F$, and the edges of $F$ are directed consistently with the given directions for the circuits of $C$. These edge directions will remain fixed. If $P$ is a circuit partition of $F$, then the circuits of $P$ can be oriented consistently with the given edge directions if and only if $P(v) \neq \psi_{C}(v)$ $\forall v\in V(F)$. Recall the notational scheme of Section 4: for each $v\in
V$, one passage of $C$ through $v$ is $h_{v}^{1},v^{+},h_{v}^{2}$ and
the other is $h_{v}^{3},v^{-},h_{v}^{4}$. 
Let $\mathcal{I}_{\mathbb{R}}(C)$ be the $V(F)\times V(F)$ matrix whose diagonal
entries all equal $0$, and whose $vw$ entry is given by: $\mathcal{I}%
_{\mathbb{R}}(C)_{vw}=1$ if $v$ and $w$ occur on $C$ in the order $v^{+}%
w^{-}v^{-}w^{+}$, $\mathcal{I}_{\mathbb{R}}(C)_{vw}=-1$ if $v$ and $w$ occur
on $C$ in the order $v^{+}w^{+}v^{-}w^{-}$, and $\mathcal{I}_{\mathbb{R}%
}(C)_{vw}=0$ if $v$ and $w$ are not interlaced on $C$. Then
\[
M_{\mathbb{R}}^{0}(C,P)=%
\begin{pmatrix}
I & \mathcal{J}_{\mathbb{R}}(C,P)\\
0 & \mathcal{I}_{\mathbb{R}}(C,P)
\end{pmatrix}
\text{,}%
\]
where $I$ is an identity matrix whose rows and columns correspond to vertices
$v\in V(F)$ with $\phi_{C}(v)=P(v)$, $\mathcal{I}_{\mathbb{R}}(C,P)$ is the
submatrix of $\mathcal{I}_{\mathbb{R}}(C)$ whose rows and columns correspond
to vertices $v\in V(F)$ with $\chi_{C}(v)=P(v)$, and $\mathcal{J}_{\mathbb{R}%
}(C,P)$ is the submatrix of $\mathcal{I}_{\mathbb{R}}(C)$ whose rows (resp.
columns) correspond to vertices $v\in V(F)$ with $\phi_{C}(v)=P(v)$ (resp.
$\chi_{C}(v)=P(v)$).

Two properties of these matrices are apparent.

\begin{itemize}
\item Both $\mathcal{I}_{\mathbb{R}}(C)$ and $\mathcal{I}_{\mathbb{R}}(C,P)$
are skew-symmetric.

\item If we interchange $v^{+}$ and $v^{-}$ on $C$, the effect on both
$\mathcal{I}_{\mathbb{R}}(C)$ and $M_{\mathbb{R}}^{0}(C,P)$ is to multiply the
$v$ row and the $v$ column by $-1$.
\end{itemize}

Some new notation will be useful. Suppose $T$ is a sub-trail of a circuit of
$C$. Let $\phi_{C}(T)\in\mathbb{Z}^{V(F)}$ be the vector whose $x$ coordinate,
for each $x\in V(F)$ with $P(x)=\phi_{C}(x)$, is obtained by tallying passages
of $T$ through $x$, with $x^{+}$ contributing $1$ and $x^{-}$ contributing
$-1$. If $P(x)=\chi_{C}(x)$ then the $x$ coordinate of $\phi_{C}(T)$ is $0$.
Let $\chi_{C}(T)\in\mathbb{Z}^{V(F)}$ be the vector obtained in the same way,
but tallying contributions only with respect to those $x$ with $P(x)=\chi
_{C}(x)$. Also, for each vertex $x\in V(F)$ let $\rho_{x}(M_{\mathbb{R}}%
^{0}(C,P))$ denote the $x$ row of $M_{\mathbb{R}}^{0}(C,P)$. The definition of
$M_{\mathbb{R}}^{0}(C,P)$ may now be rephrased as follows: if a circuit of $C$
is $x^{-}C_{1}(C,x)x^{+}C_{2}(C,x)$ then
\[
\rho_{x}(M_{\mathbb{R}}^{0}(C,P))=\chi_{C}(C_{1}(C,x))+\phi_{C}(x^{+})\text{.}%
\]

Kotzig \cite{K}, Pevzner \cite{P} and Ukkonen \cite{U} proved that if $C$ and
$C^{\prime}$ are two Euler systems of $F$, then $C$ and $C^{\prime}$ respect
the same edge directions if and only if it is possible to obtain $C^{\prime}$
from $C$ using transpositions. Consequently, in order to describe the
relationship between $M_{\mathbb{R}}^{0}(C,P)$ and $M_{\mathbb{R}}%
^{0}(C^{\prime},P)$ it suffices to understand the relationship between
$M_{\mathbb{R}}^{0}(C,P)$ and $M_{\mathbb{R}}^{0}(C\ast(vw),P)$.

\begin{proposition}
\label{transchange} Suppose the edges of $F$ are directed consistently with the circuits of $C$, and $C$ includes a circuit $v^{+}T_{1}w^{+}T_{2}v^{-}T_{3}w^{-}T_{4}$. Consider the signed
version of $C\ast(vw)$ obtained from $C$ by using $v^{+}T_{3}w^{-}T_{2}%
v^{-}T_{1}w^{+}T_{4}$. Let $P$ be a circuit partition such that $\psi
_{C}(x)\neq P(x)$ $\forall x\in V(F)$. Then $M_{\mathbb{R}}^{0}(C,P)$ and
$M_{\mathbb{R}}^{0}(C\ast(vw),P)$ are related through elementary row
operations, as follows:

\begin{enumerate}
\item $\rho_{v}(M_{\mathbb{R}}^{0}(C\ast(vw),P))=\rho_{w}(M_{\mathbb{R}}%
^{0}(C,P))$.

\item $\rho_{w}(M_{\mathbb{R}}^{0}(C\ast(vw),P))=-\rho_{v}(M_{\mathbb{R}}%
^{0}(C,P))$.

\item If $x\in V(F)-\{v,w\}$ then
\begin{gather*}
\rho_{x}(M_{\mathbb{R}}^{0}(C\ast(vw),P))\\
=\rho_{x}(M_{\mathbb{R}}^{0}(C,P))+\mathcal{I}_{\mathbb{R}}(C)_{xw}\rho
_{v}(M_{\mathbb{R}}^{0}(C,P))-\mathcal{I}_{\mathbb{R}}(C)_{xv}\rho
_{w}(M_{\mathbb{R}}^{0}(C,P)).
\end{gather*}

\end{enumerate}
\end{proposition}

\begin{proof}
Property 1 follows from Proposition \ref{pivotlabel} and the
rephrased definition of $M_{\mathbb{R}}^{0}(C,P)$ given above:
\begin{align*}
\rho_{v}(M_{\mathbb{R}}^{0}(C\ast(vw),P))  &  =\chi_{C\ast(vw)}(T_{1}%
w^{+}T_{4})+\phi_{C\ast(vw)}(v^{+})\\
&  =\chi_{C}(T_{1})+\chi_{C}(T_{4})+\phi_{C}(w^{+})+\chi_{C}(v^{+})\\
&  =\chi_{C}(T_{4}v^{+}T_{1})+\phi_{C}(w^{+})=\rho_{w}(M_{\mathbb{R}}%
^{0}(C,P))\text{.}%
\end{align*}

The proof of Property 2 uses the fact that $\sum_{i=1}^{4}\chi_{C}(T_{i})=0$:%
\begin{align*}
\rho_{w}(M_{\mathbb{R}}^{0}(C\ast(vw),P))  &  =\chi_{C\ast(vw)}(T_{2}%
v^{-}T_{1})+\phi_{C\ast(vw)}(w^{+})\\
&  =\chi_{C}(T_{2})+\chi_{C}(T_{1})+\phi_{C}(v^{-})+\chi_{C}(w^{+})\\
&  =-\chi_{C}(T_{3})-\chi_{C}(T_{4})-\phi_{C}(v^{+})-\chi_{C}(w^{-})\\
&  =-\chi_{C}(T_{3}w^{-}T_{4})-\phi_{C}(v^{+})=-\rho_{v}(M_{\mathbb{R}}%
^{0}(C,P))\text{.}%
\end{align*}

Property 3 has many cases, with $x^{-}$ and $x^{+}$ in various positions. We
detail three cases, and leave the rest to the reader.

If $x$ is not interlaced with either $v$ or $w$, then $C_{1}(C,x)$ and
$C_{1}(C\ast(vw),x)$ may not be the same trail, but they will involve the same
passages through vertices, so $\rho_{x}(M_{\mathbb{R}}^{0}(C,P))=\rho
_{x}(M_{\mathbb{R}}^{0}(C\ast(vw),P))$.

Suppose $x^{-}$ appears in $T_{1}$ and $x^{+}$ appears in $T_{2}$; say
$T_{1}=T_{11}x^{-}T_{12}$ and $T_{2}=T_{21}x^{+}T_{22}$. Then $\mathcal{I}%
_{\mathbb{R}}(C)_{xv}=0$, $\mathcal{I}_{\mathbb{R}}(C)_{xw}=1$ and%
\begin{gather*}
\rho_{x}(M_{\mathbb{R}}^{0}(C\ast(vw),P))\\
=\chi_{C\ast(vw)}(T_{12}w^{+}T_{4}v^{+}T_{3}w^{-}T_{21})+\phi_{C\ast
(vw)}(x^{+})\\
=\chi_{C}(T_{12})+\chi_{C}(T_{4})+\phi_{C}(v^{+})+\chi_{C}(T_{3})+\chi
_{C}(T_{21})+\phi_{C}(x^{+})\\
=\chi_{C}(T_{12})+\chi_{C}(w^{+}T_{21})-\chi_{C}(w^{+})+\phi_{C}(x^{+}%
)+\chi_{C}(T_{4})+\phi_{C}(v^{+})+\chi_{C}(T_{3})\\
=\chi_{C}(T_{12}w^{+}T_{21})+\phi_{C}(x^{+})+\chi_{C}(T_{3}w^{-}T_{4}%
)+\phi_{C}(v^{+})\\
=\rho_{x}(M_{\mathbb{R}}^{0}(C,P))+\rho_{v}(M_{\mathbb{R}}^{0}(C,P))\text{.}%
\end{gather*}

Suppose $x^{-}$ appears in $T_{4}$ and $x^{+}$ appears in $T_{2}$; say
$T_{2}=T_{21}x^{+}T_{22}$ and $T_{4}=T_{41}x^{-}T_{42}$. Then $\mathcal{I}%
_{\mathbb{R}}(C)_{xv}=1=\mathcal{I}_{\mathbb{R}}(C)_{xw}$ and%
\begin{gather*}
\rho_{x}(M_{\mathbb{R}}^{0}(C\ast(vw),P))=\chi_{C\ast(vw)}(T_{42}v^{+}%
T_{3}w^{-}T_{21})+\phi_{C\ast(vw)}(x^{+})\\
=\chi_{C}(T_{42})+\phi_{C}(v^{+})+\chi_{C}(T_{3})+\phi_{C}(w^{-})+\chi
_{C}(T_{21})+\phi_{C}(x^{+})\\
=\chi_{C}(T_{42}v^{+}T_{1}w^{+}T_{21})-\chi_{C}(v^{+}T_{1}w^{+})+\phi
_{C}(x^{+})-\phi_{C}(w^{+})\\
+\phi_{C}(v^{+})+\chi_{C}(T_{3})\\
=\phi_{C}(x^{+})+\chi_{C}(T_{42}v^{+}T_{1}w^{+}T_{21})+\chi_{C}(w^{-}%
)-\chi_{C}(T_{4}v^{+}T_{1})+\chi_{C}(T_{4})\\
-\phi_{C}(w^{+})+\phi_{C}(v^{+})+\chi_{C}(T_{3})\\
=\rho_{x}(M_{\mathbb{R}}^{0}(C,P))-\chi_{C}(T_{4}v^{+}T_{1})-\phi_{C}%
(w^{+})+\chi_{C}(T_{3}w^{-}T_{4})+\phi_{C}(v^{+})\\
=\rho_{x}(M_{\mathbb{R}}^{0}(C,P))-\rho_{w}(M_{\mathbb{R}}^{0}(C,P))+\rho
_{v}(M_{\mathbb{R}}^{0}(C,P))\text{.}%
\end{gather*}

\end{proof}

Proposition \ref{transchange} uses the same set of elementary row operations to
obtain $M_{\mathbb{R}}^{0}(C\ast(vw),P)$ from $M_{\mathbb{R}}^{0}(C,P)$, for
every circuit partition $P$ with $\psi_{C}(x)\neq P(x)$ $\forall x\in V(F)$.
This lack of dependence on $P$ leads to strong naturality properties, just as
it does in the proof of Corollary \ref{kappachange}. We believe these
properties have not appeared in the literature, except for the special case of
$M_{\mathbb{R}}^{0}(C^{\prime},C)=M_{\mathbb{R}}^{0}(C,C^{\prime})^{-1}$
involving Euler circuits with $\phi_{C}(v)\neq\phi_{C^{\prime}}(v)$ $\forall
v\in V(F)$, which is due to Bouchet \cite{Bu}.

\begin{corollary}
\label{orchange}Suppose $C$ and $C^{\prime}$ are Euler systems of $F$, whose circuits are oriented consistently with the same edge directions. Then for each signed version of $C$ there is a corresponding signed version of $C^{\prime}$ such that $M_{\mathbb{R}}^{0}(C^{\prime},C)=M_{\mathbb{R}}^{0}(C,C^{\prime})^{-1}$. Moreover if $P$ is
a circuit partition that respects the same edge directions, then these signed
versions of $C$ and $C^{\prime}$ have $M_{\mathbb{R}}^{0}(C^{\prime
},P)=M_{\mathbb{R}}^{0}(C^{\prime},C)\cdot M_{\mathbb{R}}^{0}(C,P)$.
\end{corollary}

\begin{proof}
According to the theorem of Kotzig \cite{K}, Pevzner \cite{P} and Ukkonen
\cite{U} mentioned above, there is a sequence of transpositions that
transforms a signed version of $C$ into a signed version of $C^{\prime}$ using
the sign convention of Proposition \ref{transchange}. Proposition
\ref{transchange} also gives us an induced sequence of elementary row
operations, which transforms the double matrix%
\[%
\begin{pmatrix}
I=M_{\mathbb{R}}^{0}(C,C) & M_{\mathbb{R}}^{0}(C,P)
\end{pmatrix}
\]
into the double matrix%
\[%
\begin{pmatrix}
M_{\mathbb{R}}^{0}(C^{\prime},C) & M_{\mathbb{R}}^{0}(C^{\prime},P)
\end{pmatrix}
\text{.}%
\]
It follows that if $E$ is the product of elementary matrices corresponding to
the induced elementary row operations, then $E\cdot I=M_{\mathbb{R}}%
^{0}(C^{\prime},C)$ and $E\cdot M_{\mathbb{R}}^{0}(C,P)=M_{\mathbb{R}}%
^{0}(C^{\prime},P)$. In particular, if $P=C^{\prime}$ we deduce that $E\cdot
I=M_{\mathbb{R}}^{0}(C^{\prime},C)$ and $E\cdot M_{\mathbb{R}}^{0}%
(C,C^{\prime})=M_{\mathbb{R}}^{0}(C^{\prime},C^{\prime})=I$.
\end{proof}

\begin{corollary}
\label{orchange2}Let $C$ and $C^{\prime}$ be Euler systems of $F$, whose circuits are oriented consistently with the same edge directions. Consider arbitrary signed versions of $C$ and $C^{\prime}$. Then there is a matrix $\Delta$ with the following properties.

\begin{enumerate}
\item Every diagonal entry of $\Delta$ is $\pm1$, and every other entry of
$\Delta$ is $0$.

\item $M_{\mathbb{R}}^{0}(C^{\prime},C)=\Delta\cdot M_{\mathbb{R}}%
^{0}(C,C^{\prime})^{-1}\cdot\Delta$.

\item If $P$ is any circuit partition that respects the same edge directions, then $M_{\mathbb{R}}^{0}(C^{\prime},P)=M_{\mathbb{R}%
}^{0}(C^{\prime},C)\cdot\Delta\cdot M_{\mathbb{R}}^{0}(C,P)\cdot\Delta$.
\end{enumerate}
\end{corollary}

\begin{proof}
Let $C^{\prime\prime}$ denote the signed version of $C^{\prime}$ that
corresponds to the given signed version of $C$, as in Corollary \ref{orchange}. For any circuit partition $P$ with $P(x)\neq\psi_{C}(x)$ $\forall x\in
V(F)$, $M_{\mathbb{R}}^{0}(C^{\prime},P)$ is the matrix obtained from
$M_{\mathbb{R}}^{0}(C^{\prime\prime},P)$ by multiplying by $-1$ the row and
column of $M_{\mathbb{R}}^{0}(C^{\prime\prime},P)$ corresponding to each $x\in
V(F)$ such that the positions of $x^{-}$ and $x^{+}$ in $C^{\prime}$ and
$C^{\prime\prime}$\ are different. Consequently if $\Delta$ is the diagonal
matrix whose $xx$ entry is $1$ (resp. $-1$) when the positions of $x^{-}$ and
$x^{+}$ in $C^{\prime}$ and $C^{\prime\prime}$ are the same (resp. different),
then $M_{\mathbb{R}}^{0}(C^{\prime},P)=\Delta\cdot M_{\mathbb{R}}%
^{0}(C^{\prime\prime},P)\cdot\Delta$. Assertions (c) and (d) now follow from
Corollary \ref{orchange}:
\[
M_{\mathbb{R}}^{0}(C^{\prime},C)=\Delta\cdot M_{\mathbb{R}}^{0}(C^{\prime
\prime},C)\cdot\Delta=\Delta\cdot M_{\mathbb{R}}^{0}(C,C^{\prime\prime}%
)^{-1}\cdot\Delta=\Delta\cdot M_{\mathbb{R}}^{0}(C,C^{\prime})^{-1}\cdot\Delta
\]
\begin{align*}
M_{\mathbb{R}}^{0}(C^{\prime},P)  &  =\Delta\cdot M_{\mathbb{R}}^{0}%
(C^{\prime\prime},P)\cdot\Delta=\Delta\cdot M_{\mathbb{R}}^{0}(C^{\prime
\prime},C)\cdot M_{\mathbb{R}}^{0}(C,P)\cdot\Delta\\
&  =M_{\mathbb{R}}^{0}(C^{\prime},C)\cdot\Delta\cdot M_{\mathbb{R}}%
^{0}(C,P)\cdot\Delta
\end{align*}
\end{proof}

Lauri \cite{Lau} and Macris and Pul\'{e} \cite{MP} gave a formula for the number
of Euler systems of $F$ that respect the same edge directions. We close with a
quick explanation of this important result.

\begin{lemma}
\label{detone}Suppose $C$ and $C^{\prime}$ are Euler systems of $F$, whose circuits are oriented consistently with the same edge directions. Then for any signed versions of $C$ and
$C^{\prime}$,%
\[
\det M_{\mathbb{R}}^{0}(C,C^{\prime})=1\text{.}%
\]

\end{lemma}

\begin{proof}
Suppose first that $C^{\prime}=C\ast(vw)$ and the signed versions of $C$ and
$C^{\prime}$ are related as in Proposition \ref{transchange}. Then Proposition
\ref{transchange} tells us how to obtain $M_{\mathbb{R}}^{0}(C^{\prime
},C^{\prime})=I$ from $M_{\mathbb{R}}^{0}(C,C^{\prime})$. The determinant is
not affected by the row operations of part 3 of Proposition \ref{transchange},
and the row operations of parts 1 and 2 -- interchanging the $v$ and $w$ rows,
and multiplying one of these rows by $-1$ -- both have the effect of
multiplying the determinant by $-1$. We conclude that in this case $\det
M_{\mathbb{R}}^{0}(C,C^{\prime})=\det I=1$.

If some other signed versions of $C$ and $C^{\prime}$ are used, then the
effect is to replace $M_{\mathbb{R}}^{0}(C,C^{\prime})$ with $\Delta\cdot
M_{\mathbb{R}}^{0}(C,C^{\prime})\cdot\Delta$, as in the proof of Corollary
\ref{orchange2}. As $\det\Delta=\pm1$, this replacement does not change the determinant.

The general case follows from part 3 of Corollary \ref{orchange2} by
induction, because $C^{\prime}$ can be obtained from $C$ using a sequence
of transpositions.
\end{proof}

\begin{corollary}
\label{detzero}Let $C$ be an Euler system of $F$, and $P$ a circuit partition
with $\psi_{C}(v)\neq P(v)$ $\forall v\in V(F)$. Then the following conditions
are equivalent.

\begin{enumerate}
\item $P$ is an Euler system.

\item $\det M_{\mathbb{R}}^{0}(C,P)=1$.

\item $\det M_{\mathbb{R}}^{0}(C,P)\not =0$.

\item $\det\mathcal{I}_{\mathbb{R}}(C,P)=1$.

\item $\det\mathcal{I}_{\mathbb{R}}(C,P)\not =0$.
\end{enumerate}
\end{corollary}

\begin{proof}
Lemma \ref{detone} gives us the implication $1\Rightarrow2$. The equality%
\[
M_{\mathbb{R}}^{0}(C,P)=%
\begin{pmatrix}
I & \mathcal{J}_{\mathbb{R}}(C,P)\\
0 & \mathcal{I}_{\mathbb{R}}(C,P)
\end{pmatrix}
\]
tells us that $\det M_{\mathbb{R}}^{0}(C,P)=\det\mathcal{I}_{\mathbb{R}}(C,P)$, so we
have $2\Leftrightarrow4$ and $3\Leftrightarrow5$. The implication
$2\Rightarrow3$ is obvious. \ According to Theorem \ref{main}, condition 3
implies that every edge of $Tch(P)$ is a loop; this in turn implies that $P$
is an Euler system.
\end{proof}

\begin{theorem}
\label{detcount}(Lauri \cite{Lau} and Macris and Pul\'{e} \cite{MP}) Let $C$ be any signed version of any Euler system of $F$. Then the number of Euler systems of $F$ that respect the edge directions given by $C$ is $\det(I+\mathcal{I}_{\mathbb{R}}(C))$.
\end{theorem}

\begin{proof}
Let $v_{1},...,v_{n}$ be the vertices of $F$, and let $x_{1},...,x_{n}$ be
independent indeterminates. For each subset $S\subseteq\{1,...,n\}$, let
$P_{S}$ be the circuit partition of $F$ that involves $\phi_{C}(v_{i})$
whenever $i\in S$, and $\chi_{C}(v_{i})$ whenever $i\notin S$. Let%
\[
\mathcal{E}=\{S\subseteq\{1,...,n\}\mid P_{S}\text{ is an Euler system of
}F\}\text{.}%
\]
Let $X$ be the matrix with entries $x_{1},...,x_{n}$ on the diagonal, and
other entries $0$. Then Corollary \ref{detzero} tells us that
\[
\det(X+\mathcal{I}_{\mathbb{R}}(C))=\sum_{S\subseteq\{1,...,n\}}\left(
\prod_{i\in S}x_{i}\right)  \det M_{\mathbb{R}}^{0}(C,P_{S})=\sum
_{S\in\mathcal{E}}\left(  \prod_{i\in S}x_{i}\right)  \text{.}%
\]
That is, $\det(X+\mathcal{I}_{\mathbb{R}}(C))$ is a version of the indicator
function of the set $\mathcal{E}$. The theorem follows by setting
$x_{1},...,x_{n}$ equal to $1$.
\end{proof}

Theorem \ref{detcount} implies that in polynomial time, one can calculate the
number of Euler systems of $F$ that respect the edge directions defined by
$C$. Ge and \v{S}tefankovi\v{c} \cite{GS} proved that in contrast, the problem
of counting all the Euler systems of $F$ is $\#P$-complete.

\end{document}